\documentclass[preprint,10pt,3p]{elsarticle}
\usepackage{lineno,hyperref}

\usepackage{graphics} 
\usepackage{epsfig} 
\usepackage{times} 
\usepackage{amsmath} 
\usepackage{amssymb}  
\usepackage{amsfonts}
\usepackage{epstopdf}
\usepackage{subfigure,color,balance}
\usepackage{verbatim}
\usepackage{cases}
\usepackage{enumerate}
\usepackage{balance}
\usepackage{framed}
\usepackage{booktabs}
\usepackage{color}
\usepackage{cancel}
\usepackage{epstopdf}
\usepackage{booktabs}
\usepackage{multirow}
\usepackage{makecell}
\usepackage{bm}
\usepackage{diagbox}

\graphicspath{{figures/}}

\newcommand{\etal}{\textit{et al}.}
\newcommand{\ie}{\textit{i}.\textit{e}.}
\newcommand{\eg}{\textit{e}.\textit{g}.}

\newtheorem{definition}{Definition}
\newdefinition{remark}{Remark}
\newtheorem{lemma}{Lemma}
\newtheorem{theorem}{Theorem}

\newproof{proof}{Proof}

\journal{Transportation Research Part C: Emerging Technologies}

\biboptions{authoryear}

\begin{document}
	
	\begin{frontmatter}
		
		\title{Mixed platoon control of automated and human-driven vehicles at a signalized intersection: dynamical analysis and optimal control}
		
		
		
		\author[a,b]{Chaoyi Chen}
		\author[a,b]{Jiawei Wang}
		\author[a,b]{Qing Xu}
		\author[a,b]{Jianqiang Wang}
		\author[a,b]{Keqiang Li\corref{cor1}}
		\ead{likq@tsinghua.edu.cn}
		
		\cortext[cor1]{Corresponding author: Keqiang Li}
		

		\address[a]{School of Vehicle and Mobility, Tsinghua University, Beijing, 100084, China}
		\address[b]{Tsinghua University-Didi Joint Research Center for Future Mobility,	Beijing, 100084, China}
		\begin{abstract}
			The emergence of Connected and Automated Vehicles (CAVs) promises better traffic mobility for future transportation systems. Existing research mostly focused on fully-autonomous scenarios, while the potential of CAV control at a mixed traffic intersection where human-driven vehicles (HDVs) also exist has been less explored. This paper proposes a notion of ``$ 1+n $" mixed platoon, consisting of one leading CAV and $ n $ following HDVs, and formulates a platoon-based optimal control framework for CAV control at a signalized intersection. Based on the linearized dynamics model of the ``$ 1+n $" mixed platoon, fundamental properties including stability and controllability are under rigorous theoretical analysis. Then, a constrained optimal control framework is established, aiming at improving the global traffic efficiency and fuel consumption at the intersection via direct control of the CAV. A hierarchical event-triggered algorithm is also designed for practical implementation of the optimal control method between adjacent mixed platoons when approaching the intersection. Extensive numerical simulations at multiple traffic volumes and market penetration rates validate the greater benefits of the mixed platoon based method, compared with traditional trajectory optimization methods for one single CAV.
		\end{abstract}
		
		
	\end{frontmatter}
	
	
	\section{Introduction}
	
	As planned points of conflict in urban traffic networks, intersections play a critical role in traffic mobility optimization. 
	Existing research has shown that the frequent stop-and-go and idling behavior of individual vehicles when approaching the intersection is the main cause of traffic congestion and casualties~\cite{lee2012development, rakha2001estimating}. Accordingly, trajectory optimization for individual vehicles at the intersection has attracted significant attention.
	In particular, the emergence of Connected and Automated Vehicles (CAVs) has provided new opportunities for improving the traffic performance at intersections. Compared to traditional human-driven vehicles (HDVs), CAVs can acquire accurate information of surrounding traffic participants and traffic signal phase and timing (SPAT) through vehicle-to-vehicle and vehicle-to-infrastructure communication~\cite{contreras2017internet}, and their velocity trajectories in approaching intersections can be directly optimized in the pursuit of higher traffic efficiency and lower fuel consumption. Multiple methods have been applied to CAV control at signalized intersections, including model predictive control~\cite{asadi2011predictive,yang2017eco}, fuzzy logic~\cite{Milanes10,onieva2015multi} and optimal control~\cite{li2015eco,jiang2017eco}. Moreover, the potential of cooperative control of traffic signals and CAVs has also been recently discussed~\cite{xu2018cooperative}.
	The aforementioned research mainly considered a fully-autonomous scenario---the market penetration rate (MPR) of CAVs is 100\%. In practice, however, it might take decades for all the HDVs in current transportation systems to be transformed into CAVs. Instead, a more practical scenario in the near future is a mixed traffic system where CAVs and HDVs coexist~\cite{zheng2020smoothing}. Existing research on mixed traffic intersections mostly focused on the estimation of traffic states and optimization of traffic signals~\cite{Priemer09, Salman17, zheng2017estimating, feng2015real}. For example, Priemer \etal~utilized dynamic programming to estimate the queuing length in the mixed traffic environment and then optimized the traffic phase time based on the estimated results~\cite{Priemer09}. Some other methods, including fuzzy logic~\cite{Salman17}, expectation maximization~\cite{zheng2017estimating} and dynamic programming~\cite{feng2018spatiotemporal}, have also been exploited to achieve a similar goal. To improve the estimation accuracy, the celebrated results on microscopic car-following models have been recently employed to describe the behaviors of HDVs, including Gipps Model~\cite{gipps1981behavioural}, Optimal Velocity Model (OVM)~\cite{helbing1998generalized} and Intelligent Driver Model (IDM)~\cite{treiber2013traffic}; see, \eg,~\cite{du2017coordination,fang2020trajectory}.
	
	Despite these existing works, the topic of CAV control, \ie, trajectory optimization of CAVs, in mixed traffic intersections has not been fully discussed. To tackle this problem, several works regarded HDVs as disturbances in the control of CAVs~\cite{du2017coordination,ala2016modeling}, or focused on the task of collision avoidance based on the prediction of HDVs' behaviors~\cite{jiang2017eco,ma2017parsimonious}. It is worth noting that most of these methods were limited to improving the performance of CAVs themselves in their optimization frameworks, instead of optimizing the global traffic flow consisting of both HDVs and CAVs at the intersection. Two notable exceptions are in~\cite{zhao2018platoon,liang2019joint}, which attempted to improve the performance at the signalized intersection from the perspective of the so-called \emph{mixed platoon}. They enumerated several possible formations consisting of HDVs and CAVs, and investigated their effectiveness through small-scale simulation experiments. However, a general and explicit definition of the mixed platoon has not been clarified, and fundamental properties of the mixed platoon at intersections have been less explored. Moreover, a specific optimal control framework for the mixed platoon at intersections with global consideration of improving the entire traffic performance is still lacking.
	
	{In fact, considering the interaction between adjacent vehicles on the same lane, it is easy to understand that velocity trajectories of CAVs could have a certain influence on those of surrounding vehicles, especially the vehicles following behind them. Accordingly, the driving strategies of CAVs could have a direct impact on the performance of the entire mixed traffic intersection; such impact might even be negative when inappropriate CAV strategies are adopted~\cite{ala2016modeling}. By contrast, when taking the performance of entire mixed traffic into explicit consideration, the optimization of CAVs' trajectories could bring further benefits to traffic mobility. Such idea of improving the global traffic performance via controlling CAVs has been recently proposed as \emph{Lagrangian Control} of traffic flow~\cite{stern2018dissipation}, which has been discussed in various traffic scenarios, including closed ring road~\cite{zheng2020smoothing}, open straight road~\cite{wang2020leading}, traffic bottleneck~\cite{vinitsky2018lagrangian}, and non-signalized intersection~\cite{wu2017emergent}. Regarding signalized intersections, the potential of this notion has not been well understood.}

	
	In this paper, we focus on the scenario of a signalized intersection where HDVs and CAVs coexist and aim at improving the performance of the entire mixed traffic intersection through direct control of CAVs. To address this problem, we propose a novel framework which separates the traffic flow into ``$1+n$" microstructures consisting of one leading CAV and $n$ following HDVs. Such microstructure is particularly common in the near future when the MPR is relatively low, which is named as ``$1+n$" mixed platoon. We discuss the possibility of letting the first CAV lead the motion of following $ n $ HDVs in approaching intersections, and show how to enable the CAV to benefit the global traffic mobility in the proposed structure of mixed platoon. First, we present the dynamics model of ``$1+n$" mixed platoon systems based on linearized car-following models and investigate its fundamental properties, including open-loop stability and controllability. Then, we establish the optimal control framework for the mixed platoon and design a hierarchical algorithm to improve the global traffic mobility performance at a signalized intersection. Specifically, our contributions are as follows: 
	
	\begin{enumerate}[(1)]
		\item The notion of the ``$1+n$" mixed platoon is proposed for trajectory optimization of CAVs at signalized intersections in the mixed traffic environment. Rather than enumerating the mixed platoon formations~\cite{zhao2018platoon,liang2019joint}, we provide an explicit definition of the mixed platoon. Based on the linearized dynamics model, we perform a rigorous theoretical analysis of its fundamental properties, including open-loop stability and controllability. Our theoretical results reveal that the ``$1+n$" mixed platoon is always controllable under a very mild condition, regardless of the platoon size $n$.
		\item An optimal control framework is established for the ``$1+n $'' mixed platoon at a constant traffic SPAT intersection scenario. Instead of being limited to optimizing the performance of the CAVs only~\cite{asadi2011predictive,yang2017eco,Milanes10,onieva2015multi,li2015eco,jiang2017eco,du2017coordination,fang2020trajectory,ala2016modeling,ma2017parsimonious}, our optimal control formulation aims at improving the global performance of the entire mixed traffic intersection. Precisely, the velocity deviations and fuel consumption of all the vehicles in the mixed platoon are under explicit consideration. Moreover, unlike~\cite{zhao2018platoon,liang2019joint}, we optimize the terminal velocity setting to maximize the traffic throughput.
		\item Finally, a hierarchical algorithm is proposed to accomplish the optimal control framework of ``$1+n $'' mixed platoons, which can be applied in any MPR of mixed traffic environments. An event-triggered mechanism is designed to avoid potential collisions of adjacent mixed platoons. Large-scale traffic simulations are conducted at multiple traffic volumes and MPRs, and it is observed that the proposed mixed platoon based control method surpasses the traditional intersection control method for single CAV~\cite{asadi2011predictive,yang2017eco} in both traffic efficiency and fuel consumption.
	\end{enumerate}
	
	
	
	The rest of this paper is organized as follows. Section~\ref{Section:ProblemFormulation} introduces the problem of the system modeling for the proposed``$1+n $'' mixed platoon. Section~\ref{Section:Methodology} presents the system dynamics analysis, optimal control framework and algorithm design. The simulation results are shown in Section~\ref{Section:SimulationResult}, and Section~\ref{Section:Conclusion} concludes this paper.
	
	\section{Problem Statement}
	\label{Section:ProblemFormulation}
	In this section, we firstly introduce the scenario setup, then present the dynamical modeling of individual vehicles and the mixed platoon systems.
	
	\begin{figure}[t]
		\centering
		\includegraphics[width=0.65\linewidth]{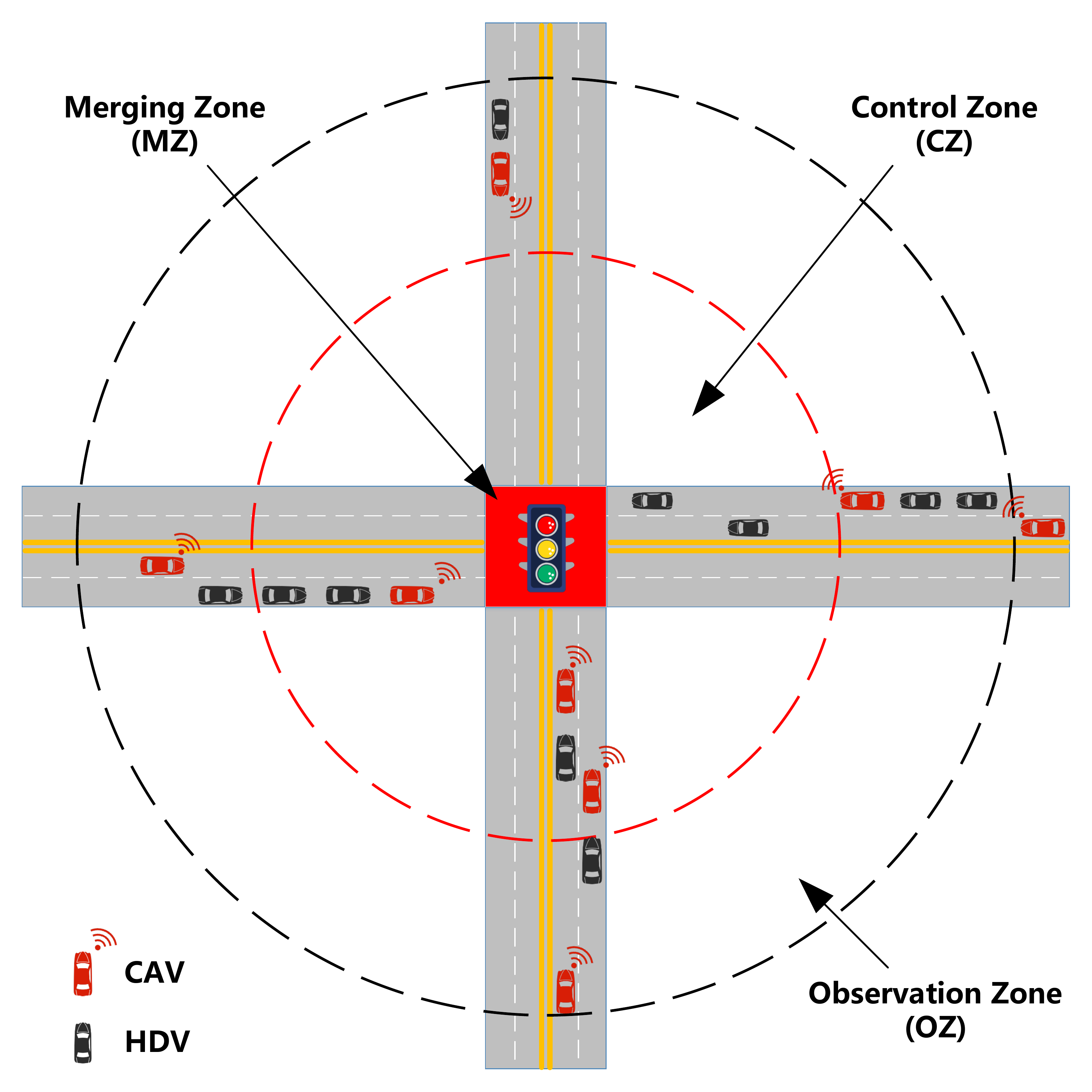}
		\caption{Illustration for the mixed traffic intersection. Red vehicles represent CAVs, which can transmit and receive vehicle information and the vehicle is fully autonomous driving. Black vehicles represent HDVs, which can only transmit ego vehicle information to other vehicles and are controlled by car following model.}
		\label{fig:IntersectionStructure}
	\end{figure}
	
	\subsection{Scenario Setup}
	\label{section:ScenarioSetup}
	
	In this paper, we consider a typical signalized intersection scenario in the mixed traffic environment, where HDVs and CAVs coexist; see Fig.~\ref{fig:IntersectionStructure} for illustration. 
	A traffic light is deployed in the center to guide individual vehicles to drive through the intersection. Note that CAVs follow the instructions from a \emph{central cloud coordinator}, which collects information from all the involved vehicles around the intersection and calculates the optimal velocity trajectories for each CAV. The design of the control strategies for the central cloud coordinator is presented in Section~\ref{Section:Methodology}.

	Motivated by previous works on signalized intersections~\cite{feng2015real,malikopoulos2018decentralized,bian2019cooperation}, we separate the intersection into three zones, as illustrated in Fig.~\ref{fig:IntersectionStructure}. The red square area in the center is named as the Merging Zone (MZ), which is the area of potential lateral collision of the involved vehicles. The ring area between two dashed lines is called the Observation Zone (OZ), where the CAVs and HDVs are allowed to perform lane-changing behaviors. The area between OZ and MZ is the Control Zone (CZ), where the CAVs are under direct control from the central cloud coordinator. The specific range of each zone is discussed in Section~\ref{Section:SimulationResult}.
	
	Similar to existing research~\cite{zhao2018platoon,liang2019joint}, the following assumptions are needed to facilitate the control design for signalized intersection controls, as well as the system modeling and dynamics analysis.
	
	\begin{enumerate}[(1)]
		\item All the vehicles are connected vehicles, which means that both the CAVs and the HDVs are able to transmit their velocity and position to the central cloud coordinator through wireless communication, \emph{e.g.}, V2I communication~\cite{gerla2014internet}. An ideal communication condition without communication delay or packet loss is under consideration.
		
		\item All the CAVs are capable of fully autonomous driving, which follow the velocity trajectories assigned from the central cloud coordinator after they enter CZ. Regarding the HDVs, they are controlled by human drivers, for which we assume a general car-following model to describe their driving behavior (see Section~\ref{Sec:HDVModel} for details).  
		
		\item {Lane changing is not permitted in CZ. As shown in~\cite{ala2016modeling, xu2018distributed}, unexpected lane changing behaviors might worsen the traffic efficiency, especially near the intersection. Therefore, lane changing is only allowed in OZ, while In CZ, we only need to focus on the longitudinal behavior of each vehicle.}
	\end{enumerate}

	\subsection{Dynamical Modeling of Mixed Platoon Systems}
	\label{Sec:HDVModel}
	
	\begin{figure}[!t]
		\centering
		\includegraphics[width=0.7\linewidth]{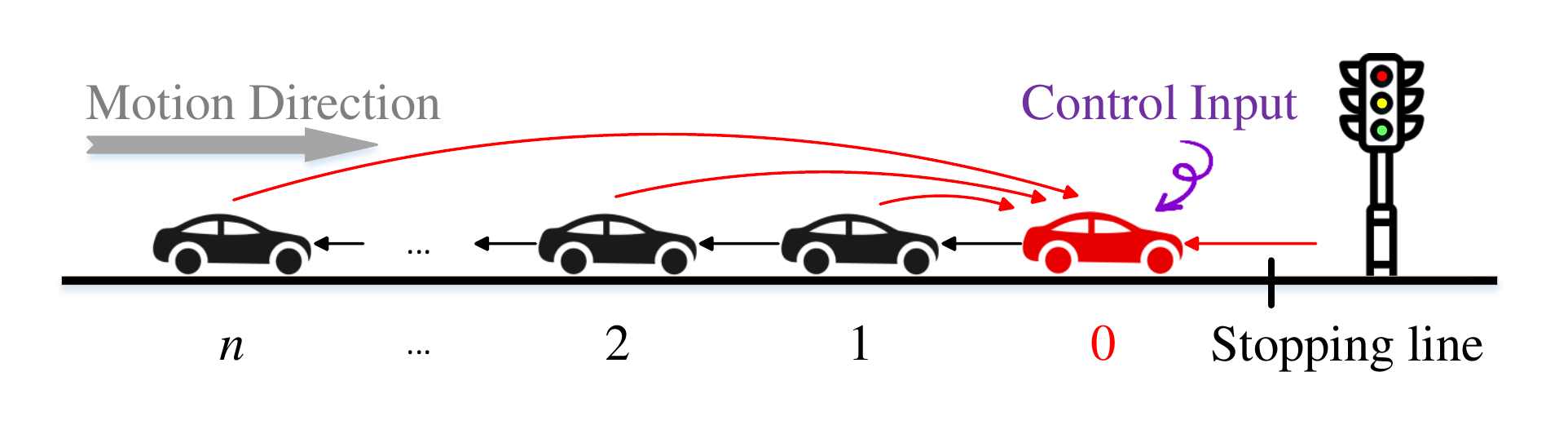}
		\caption{Schematic for the ``$1+n$'' mixed platoon. Red arrows denote the information flow of the leading CAV (colored in red), which collects information from the following HDVs as well as the traffic light and has an external control input. Black arrows represent the information flow of the HDVs (colored in black), which are under control by human drivers and only acquire information from the preceding vehicle, as shown in~\eqref{equ:CFmodel}. {Note that the coordinate origin is set at the stopping line, where the vehicle position is $ x = 0 $.}}
		\label{fig:InformationTopology}
	\end{figure}

	
	In this paper, we propose the notion of ``$1+n$'' mixed platoon as shown in Fig.~\ref{fig:InformationTopology}. It consists of one leading CAV and $n$ following HDVs. Specifically, each CAV is designed as the leading vehicle of the mixed platoon, which leads the motion of the following $n$ HDVs with the aim of improving the performance of the entire mixed platoon when passing the intersection.
	

	Consider the $n$ following HDVs in the mixed platoon illustrated in Fig.~\ref{fig:InformationTopology}. We denote the position and velocity of vehicle $i$ at time $t$ as $x_i(t)$ and $v_i(t)$, respectively. The headway distance of vehicle $i$ from vehicle $i-1$ is defined as ${d}_{i}(t):=x_{i-1}(t)-x_{i}(t) $. Then, $ \dot{v}_{i}(t) $ and $\dot{d}_{i}(t) =v_{i-1}(t)-v_{i}(t)$ represents the acceleration of vehicle $ i $,  and the relative velocity of vehicle $i$ from vehicle $i-1$, respectively.
	
	A great many efforts have been made in previous works to describe the HDVs' car-following dynamics, with several significant models developed, \emph{e.g.}, OVM~\cite{helbing1998generalized} and IDM~\cite{treiber2013traffic}. As shown in the literature~\cite{zheng2020smoothing,wilson2011car}, most of them can be expressed as the following general form ($i=1,\ldots,n$)
	\begin{equation}
		\label{equ:CFmodel}
		\dot{v}_{i}(t) =F\left(d_{i}(t), \dot{d}_{i}(t), v_{i}(t)\right),
	\end{equation}
	which means that the acceleration of vehicle $i$ is determined by its headway distance, relative velocity and its own velocity. 
	
	In this paper, we require that the mixed platoon passes the intersection at a pre-specified equilibrium velocity $v^{\ast}$. In equilibrium traffic state, we have $ \dot{d}_{i}(t) = 0 $ for $i=1,\ldots,n$, and thus each vehicle has a corresponding equilibrium headway distance $ d^{\ast} $, where it holds that
	\begin{equation}
		\label{equ:CFmodel:Stable}
		F\left(d^{\ast}, 0, v^{\ast}\right) = 0.
	\end{equation}
	Then we employ the deviation of the current state $(d_i(t),v_i(t))$ of vehicle $i$ from the equilibrium state $(d^{\ast},v^{\ast})$ as its state variable, given by
	\begin{equation}
		\begin{cases} 
			\label{equ:ErrorState}
			\tilde{d}_{i}(t) =d_{i}(t)-d^{*}, \\ 
			\tilde{v}_{i}(t) =v_{i}(t)-v^{*}. 
		\end{cases}
	\end{equation}
	
	Applying the first-order Taylor expansion to~\eqref{equ:CFmodel} leads to the linearized dynamics of HDVs around the equilibrium state $(d^{\ast},v^{\ast})$, given as follows ($i=1,\ldots,n$)
	\begin{equation}
		\label{equ:ErrorSpeedDiff}
		\begin{cases} 
			\dot{\tilde{d}}_i(t)= v_{i-1}(t)-v_i(t),\\
			\dot{\tilde{v}}_{i}(t) =\alpha_{1} \tilde{d}_{i}(t)-\alpha_{2} \tilde{v}_{i}(t)+\alpha_{3} \tilde{v}_{i-1} (t),
		\end{cases}
	\end{equation}
	with
	$
	\label{equ:ErrorSpeedDiff2}
	\alpha_{1}=\frac{{\partial} F}{{\partial} d},\,
	\alpha_{2}=\frac{{\partial} F}{{\partial} \dot{d}}-\frac{{\partial} F}{{\partial} v},\, 
	\alpha_{3}=\frac{{\partial} F}{{\partial} \dot{d}},
	$
	evaluated at the equilibrium state $(d^{\ast},v^{\ast})$. For trade-off between model fidelity and computational tractability, we consider the OVM model~\cite{helbing1998generalized} as the specific car-following model in our optimal control formulation in Section~\ref{Section:OptimalControlMode}. Note that other HDV models, \eg, IDM~\cite{treiber2013traffic}, can also be applied to derive the specific expression for~\eqref{equ:ErrorSpeedDiff}. In OVM, the general expression~\eqref{equ:CFmodel} of HDVs' car-following dynamics becomes
		\begin{equation}
			\label{equ2:OVM_a}
			\dot{v}_{i}=\kappa[V_{\mathrm{des}}({d}_{i})-v_i],
		\end{equation}
		where $ V_{\mathrm{des}} $ is the driver's desired velocity at headway distance $ {d}_{i} $, given by
		\begin{equation}
			\label{equ2:OVM_vopt}
			V_{\mathrm{des}}({d}_{i})=V_1+V_2{\tanh}[C_1({d}_{i}-L_\mathrm{veh})-C_2],
		\end{equation}
		with $ L_{\mathrm{veh}} $ denoting the vehicle length and the rest of the symbols are constants. In this case, the specific value of the coefficients in~\eqref{equ:ErrorSpeedDiff} can be calculated by
		\begin{equation}
			\label{equ2:OVM_vopt_alphai}
			\alpha_{1} = {\kappa}V_2C_1\left\{1-{\tanh}^2[C_1({d}_{i}-L_\mathrm{veh})-C_2]\right\}, \alpha_{2} = \kappa, \alpha_{3} = 0.
		\end{equation}
	
	Regarding the leading CAV, indexed as vehicle $0$, its acceleration signal is utilized as the control input $u(t)$. Then, the longitudinal dynamics of the leading vehicle can be expressed as the following second-order form
	\begin{equation}
		\begin{cases} 
			\label{equ:CAV_State}
			\dot{x}_{0}(t) =v_{0}(t), \\ 
			\dot{v}_{0}(t) =u(t) .
		\end{cases}
	\end{equation}
	
	It is worth noting the acceleration signal of the leading CAV is also the only external control input of the entire system of the mixed platoon illustrated in Fig.~\ref{fig:InformationTopology}. 
	Lumping the state of both the leading CAV and the following HDVs yields the state vector of the entire mixed platoon system, given as follows
	\begin{equation}
		\label{equ:Platoon_state}
		{X}(t)
		=\left[
		\begin{matrix}
			x_0(t) & v_0(t) & \tilde{d}_1(t) & \tilde{v}_1(t) &\dots & \tilde{d}_{n}(t) & \tilde{v}_{n}(t)
		\end{matrix}
		\right]^ {\rm T}.
	\end{equation}
	
	Based on~\eqref{equ:ErrorSpeedDiff} and~\eqref{equ:CAV_State}, the state-space model of the mixed platoon system is obtained
	\begin{equation}
		\label{equ:MixedPlatoon}
		\dot{{X}}(t)={A} {X}(t)+{B} {u}(t),
	\end{equation}
	where (${A} \in \mathbb{R}^{(2n+2) \times(2n+2)}, \, {B} \in \mathbb{R}^{(2n+2) \times 1}$)
	\begin{equation*}
		\label{equ:MatrixA}
		{A}=\left[ \begin{array}{cccccc}
			{{C}_{1}} & {0} & {\cdots} & {\cdots} & {0} & 0 \\ 
			{{A}_{2}} & {{A}_{1}} & {0} & {\cdots} & {\cdots} & {0} \\ 
			{0} & {{A}_{2}} & {{A}_{1}} & {0} & {\cdots} & {0} \\ 
			{\vdots} & {\ddots} & {\ddots} & {\ddots} & {\ddots} & {\vdots} \\ 
			{0} & {\cdots} & {0} & {{A}_{2}} & {{A}_{1}} & {0} \\ 
			{0} & {\cdots} & {\cdots} & {0} & {{A}_{2}} & {{A}_{1}}
		\end{array}\right],\;
		{B}=\left[ \begin{array}{c}
			B_1\\B_2\\B_2\\ \vdots \\B_2\\B_2
		\end{array}\right],
	\end{equation*}
	with
	\begin{equation*}
		\label{equ:MatrixC}
		{A}_{1}=\left[ \begin{array}{cc}
			{0} & {-1} \\ 
			{\alpha_{1}} & {-\alpha_{2}}
		\end{array}\right], \,
		{A}_{2}=\left[ \begin{array}{cc}
			{0} & {1} \\ 
			{0} & {\alpha_{3}}
		\end{array}\right],\,
		{B}_{1}=\left[ \begin{array}{c}
			{0} \\ 
			1
		\end{array}\right],\,
		{B}_{2}=\left[ \begin{array}{c}
			{0} \\ 
			0
		\end{array}\right],\,
		{C}_{1}=\left[ \begin{array}{cc}
			{0} & {1} \\ 
			{0} & {0}
		\end{array}\right].
	\end{equation*}
	
	
	
	
	
	

	\begin{remark}
		Previous work on control of individual vehicles at signalized intersections mostly focused on the dynamics~\eqref{equ:CAV_State} of CAVs only; see, \eg,~\cite{yang2017eco,li2015eco, jiang2017eco}. Considering the interaction between neighboring vehicles, in this paper we focus on the dynamics~\eqref{equ:MixedPlatoon} of the entire mixed platoon system consisting of both the CAV and its following HDVs, and seek to improve the overall performance via controlling the CAV only. A similar idea has been recently proposed by Stern \etal~\cite{stern2018dissipation}, known as \emph{Lagrangian Control} of traffic flow, where CAVs are utilized as \emph{mobile actuators} to control the entire mixed traffic system. The effectiveness of this notion has been validated in various scenarios, including closed ring road~\cite{zheng2020smoothing}, open straight road~\cite{wang2020leading}, traffic bottleneck~\cite{vinitsky2018lagrangian}, and nonsignalized intersection~\cite{wu2017emergent}. To the best of our knowledge, the feasibility of this notion has not been explicitly discussed in the scenario of signalized intersections.
	\end{remark}
	
	\section{Methodology}
	\label{Section:Methodology}
	Firstly we analyze the open-loop stability and controllability of the proposed ``$ 1+n $'' mixed platoon systems. Based on the derived stability and controllability conditions, an optimal control framework is proposed to optimize the CAVs' driving strategy in a single mixed platoon. Finally, an event-triggered algorithm is established to solve the collision problem of different mixed platoons.
	
	
	\subsection{Open-Loop Stability Analysis}
	\label{Section:FundamentalProperties:Stability}
	
	Based on the model~\eqref{equ:Platoon_state} of the mixed platoon system, we first consider its open-loop stability as shown in Definition~\ref{def:Stability}, when the leading CAV has no external control input, \ie, $u(t)=0$.
	
	{	\begin{definition}[Lyapunov Stability~\cite{skogestad2007multivariable}]
			\label{def:Stability}
			For a general dynamical system $\dot{x} = f(x(t))$, the equilibrium $x_e$ is said to be Lyapunov stable, if $\forall \epsilon>0$, there exists a $\delta >0$ such that, if $\Vert x(0) - x_e \Vert < \delta$, then for every $t\geq 0$ we have $\Vert x(t) - x_e \Vert < \epsilon$.
		\end{definition}
		Typically, the stability of a nonlinear system $\dot{x} = f(x(t))$ around an equilibrium point is analyzed after system linearization, and the linearized system  $\dot{x} = Ax(t) $ is (asymptotically) stable if and only if all the eigenvalues of $A$ have negative real parts.} Existing research have revealed the stability criterion of the linearized car-following model~\eqref{equ:ErrorSpeedDiff} of one single HDV, shown in Lemma~\ref{equ2:Steady}. 
	
	\begin{lemma}[\cite{wilson2011car}]
		\label{equ2:Steady}
		The linearized car-following model~\eqref{equ:ErrorSpeedDiff} is stable if and only if $\alpha_{1}>0,\alpha_{2}>0$.
	\end{lemma}
	
	Regarding the ``$ 1+n $'' mixed platoon system, where there exist $n$ following HDVs, we have the following result. 
	\begin{theorem} 
		\label{Th:Stability}
		The ``$1+n$'' mixed platoon system~\eqref{equ:MixedPlatoon} is open-loop stable if and only if $\alpha_{1}>0, \alpha_{2}>0$, which is irrelevant to the platoon size $n$.
	\end{theorem}
	
	%
	
	{The proof of Theorem~\ref{Th:Stability} is shown in~\ref{sec:Appendix}. Note that there always exist two zero eigenvalues in $A$, as shown in~\eqref{equ2:n2:EigenValue}, and thus the mixed platoon system~\eqref{equ:MixedPlatoon} is Lyapunov stable, but not asymptotically stable. It can be easily seen that the zero eigenvalues are brought by the states of the leading CAV in the open-loop case. Indeed, when $\alpha_{1}>0, \alpha_{2}>0$, the subsystem consisting of the states of the following HDVs is strictly asymptotically stable, while the mixed platoon system is Lyapunov stable.}
	
	{Recall that the specific value of the linear coefficients $\alpha_i \,(i=1,2,3)$ under the OVM model~\eqref{equ2:OVM_a} has been derived in~\eqref{equ2:OVM_vopt_alphai}. According to Theorem~\ref{Th:Stability}, it is straightforward to obtain the stability condition for the OVM model after linearization: $ {\kappa} > 0, V_2C_1 > 0 $. Note that string stability is also an important topic for CAV's longitudinal control, which describes the propagation of the perturbations in a string of vehicles. Existing research for string stability of mixed platoons mostly focuses on one special case where one CAV is following at the tail behind a string of HDVs; see, \ie, \cite{wu2018stabilizing,jin2014dynamics}. Interested readers are referred to~\cite{wang2020leadingSCS} for further analysis on the proposed ``$ 1+n $" mixed platoon system and more general cases.}
	
	\subsection{Controllability Analysis}
	\label{Section:FundamentalProperties:Controllability}
	
	One goal of the control of the mixed platoon is to pass the intersection with a pre-specified equilibrium velocity, and controllability is a fundamental property to depict the feasibility of this goal. Particularly, if controllability holds for the mixed platoon, then the mixed platoon system can be moved to any desired state under the control input of the CAV. The formal definition and one useful criterion of controllability are shown as follows. 
	
	
	\begin{definition}[Controllability~\cite{skogestad2007multivariable}]
		\label{def:Controllability}
		The dynamical system $ \dot{x}(t)=Ax(t)+Bu(t) $, or the pair $ (A,B) $, is controllable if and only if, for any initial state $ x(0)=x_0 $, any time $ t_f>0 $ and any final state $ x_f $, there exists an input $ u(t) $ such that $ x(t_f)=x_f $.
	\end{definition}

	\begin{lemma}[Popov-Belevitch-Hautus criterion~\cite{skogestad2007multivariable}]
		\label{equ2:Control:PBH}
		In a continuous-time LTI system $({A},{B})$ of size $n$, the system is controllable if and only if for every eigenvalue $ \lambda $, $ {\rm rank}\left({\lambda}{I} - {A},{B}\right) = n$.
		System $({A},{B})$ is uncontrollable if and only if there exists $ {\rho} \neq 0$, such that ${\rho}^{\top} {A}=\lambda {\rho}^{\top}$, ${\rho}^{\top} {B}=0$.
	\end{lemma}
	
	Our result regarding the controllability of the $1+n$ mixed platoon is as follows.
	\begin{theorem}
		\label{Th:Controllability}
		The ``$1+n$'' mixed platoon system~\eqref{equ:MixedPlatoon} is controllable when the following condition holds, which is irrelevant to the platoon size $n$.
		\begin{equation}
			\label{equ:ControllabilityCondition}
			\alpha_{1}-\alpha_{2} \alpha_{3}+\alpha_{3}^{2} \neq 0.
		\end{equation}
	\end{theorem}
	
	\begin{proof}
		Assume that the mixed platoon system~\eqref{equ:MixedPlatoon} is uncontrollable. According to Lemma~\ref{equ2:Control:PBH}, there exists a scalar $\lambda$ and a non-zero vector $ {\rho} = [\rho_{01},\rho_{02},\rho_{11},\rho_{12},\dots,\rho_{n1},\rho_{n2}]^ {\rm T} $, where $ \rho_{ij}\in \mathbb{R} $, which satisfy
		\begin{equation}
			\rho^ {\rm T}({A}-\lambda {I})=0, \; \rho^ {\rm T}{B} = 0.
		\end{equation}
		From $ \rho^ {\rm T}{B} = 0 $, we have $ \rho_{02} = 0 $. From $\rho^ {\rm T}({A}-\lambda {I})=0$, it is obtained that
		\begin{equation}
			\label{equ2:1:EigenVector}
			\begin{cases}
				{-\lambda \rho_{01}=0}, \\ 
				{\rho_{01}-\lambda \rho_{02}+\rho_{11}+\alpha_{3} \rho_{12}=0},
			\end{cases}
		\end{equation}
		and
		\begin{equation}
			\label{equ2:n:EigenVector}
			-\rho_{n1}-\left(\alpha_{2}+\lambda\right) \rho_{n2}=0,
		\end{equation}
		and for $i=1,\ldots,n$,
		\begin{equation}
			\label{equ2:1_n:EigenVector}
			\begin{cases}
				{-\lambda \rho_{i1}+\alpha_{1} \rho_{i2}=0}, \\ 
				{\rho_{i1}+\left(-\alpha_{2}-\lambda\right) \rho_{i2}+\rho_{(i+1)1}+\alpha_{3} \rho_{(i+1)2}=0}.
			\end{cases}
		\end{equation}
		According to~\eqref{equ2:1_n:EigenVector}, we have
		\begin{equation}
			\label{equ2:After:i}
			\left(\lambda^{2}+\alpha_{2}\lambda+\alpha_{1}\right) \rho_{i1}=\left(\alpha_{3}\lambda+\alpha_{1}\right) \rho_{(i+1)1}, \; i=1,\ldots,n,
		\end{equation}
		and when $ i = 1 $, it holds that $ \lambda\rho_{11} = \alpha_{1}\rho_{12} $. Substituting it into~\eqref{equ2:1:EigenVector} yields
		\begin{equation}
			\label{equ2:After:1}
			\left(\alpha_{3} \lambda+\alpha_{1}\right) \rho_{11}=0.
		\end{equation}
		From~\eqref{equ2:n:EigenVector}, we have
		\begin{equation}
			\label{equ2:After:n}
			\left(\lambda^{2}+\alpha_{2} \lambda+\alpha_{1}\right) \rho_{n 1}=0.
		\end{equation}
		
		It can be easily examined that when $\alpha_{1}-\alpha_{2} \alpha_{3}+\alpha_{3}^{2} \neq 0$, the two equations
		$\lambda^{2}+\alpha_{2} \lambda+\alpha_{1} = 0 $ and $ \alpha_{3}{\lambda}+\alpha_{1} = 0 $ cannot be satisfied simultaneously. Therefore, we consider the following three cases: (1)${\lambda}^{2}+\alpha_{2}\lambda+\alpha_{1} = 0$, $\alpha_{3}{\lambda}+\alpha_{1} \neq 0$; (2)${\lambda}^{2}+\alpha_{2}\lambda+\alpha_{1} \neq 0$, $\alpha_{3}{\lambda}+\alpha_{1} = 0$; (3)${\lambda}^{2}+\alpha_{2}\lambda+\alpha_{1} \neq 0$, $\alpha_{3}{\lambda}+\alpha_{1} \neq 0$. In each case, it can be obtained that $ \rho_{i1}=\rho_{i2}=0 $, $i=0,1,\ldots,n$ by combing $ \rho_{02} = 0 $ and equations~\eqref{equ2:After:i},~\eqref{equ2:After:1},~\eqref{equ2:After:n}. This contradicts the requirement that $ \rho \neq 0 $, which indicates that the assumption does not hold. Therefore, the system $({A},{B})$ is controllable when condition~\eqref{equ:ControllabilityCondition} holds.
	\end{proof}
	
	\begin{remark}
		Theorem~\ref{Th:Controllability} indicates that the ``$1+n$'' mixed platoon is controllable with regard to the control input of the leading CAV when condition~\eqref{equ:ControllabilityCondition} holds. This result indicates that through controlling the leading CAV directly, one has complete control of the motion of the following $n$ HDVs without changing their natural driving behaviors. This property allows for the feasibility of designing the control input of the single CAV with the aim of improving the performance of the entire ``$1+n$'' mixed platoon. Note that condition~\eqref{equ:ControllabilityCondition} is consistent with previous controllability analysis in~\cite{zheng2020smoothing,cui2017stabilizing}, which focused on mixed platoons in a closed ring-road traffic system. Interested readers are referred to~\cite{wang2020controllability} to make further investigations on the controllability property of the ``$1+n$'' mixed platoon when the following $n$ HDVs have heterogeneous dynamics in~\eqref{equ:CFmodel}.
	\end{remark}
	
	\subsection{Optimal Control Framework}
	\label{Section:OptimalControlMode}
	After the dynamical analysis of the fundamental properties of the proposed ``$ 1+n$'' mixed platoon, in the following section, we proceed to establish an optimal control framework for the mixed platoon at the signalized intersection.

	\subsubsection{Cost Function}
	In our optimal control framework of the ``$ 1+n$'' mixed platoon, the main control objective is to let the CAV reach the stopping line of the intersection when the traffic signal turns green, and meanwhile the following HDVs are stabilized at a desired equilibrium velocity $v^{\ast}$, as discussed in~\eqref{equ:CFmodel:Stable}. 
	Moreover, we also aim at minimizing the fuel consumption of the entire mixed platoon during its process of approaching the intersection. Accordingly, we define the following cost function in the Bolza form
	\begin{equation}
		\label{equ2:cost_fuction}
		\newcommand{\ud}{\mathrm{d}}
		J = \varphi({X}(t_\mathrm{f})) + \int_{t_0}^{t_\mathrm{f}}L({X}(t),{u}(t))\ud t,
	\end{equation}
	where $ t_0 $ is the time when CAV enters CZ, \ie, reaches the boundary of CZ and OZ as shown in Fig.~\ref{fig:Algorithm}. And $t_\mathrm{f}$ denote the time when CAV enters CZ, \ie, reaches the stopping line, which will be discussed later in Section~\ref{Section:Algorithm:MixedPlatoon}.

	
	As the terminal cost function in~\eqref{equ2:cost_fuction}, $\varphi({X}(t_f))$ measures the deviation of the system final state from the desired state, which is defined as
	\begin{equation}
		\label{equ2:ternimal_cost_fuction}
		\varphi\big({X}(t_\mathrm{f})\big) = \omega_1\big(x_0(t_\mathrm{f}\big) - x_{\mathrm{tar}})^2 + \omega_2\sum_{i=0}^{n}\big(v_i(t_\mathrm{f}) - v^{\ast}\big)^2,
	\end{equation}
	where $ \omega_1 $ and $ \omega_2 $ denote the weight coefficients for penalty of the position deviation of the leading CAV and the velocity deviation of all the vehicles in the mixed platoon, respectively. $ x_0(t_\mathrm{f}) $ denotes the position of the leading CAV at $ t=t_\mathrm{f} $. $ x_{\mathrm{tar}} $ is the target final position of the leading CAV, which refers to the position of the stopping line at the intersection. The specific choice of the desired equilibrium velocity $ v^{\ast} $ and the target position $ x_{\mathrm{tar}} $ of the CAV will be discussed in Sections~\ref{Sec:Terminal Speed} and~\ref{Sec:Constraints}, respectively. 
	
	In~\eqref{equ2:cost_fuction}, $L\big({X}(t),{u}(t)\big)$ denotes the transient fuel consumption of the mixed platoon at time $t$, which is defined as
	\begin{equation}
		\label{equ2:Fuel_consumption_all}
		L\big({X}(t), {u(t)}\big)=G_{0}(t)+\sum_{i = 1}^{n} G_{i}(t),
	\end{equation}
	where $G_{0}(t)$ and $G_{i}(t)$, $i=1,\ldots,n$ represent the transient fuel consumption of the leading CAV and the following HDVs, respectively. Similar to recent work~\cite{zhao2018platoon, jiang2017eco}, we utilize the Akcelik's fuel consumption model for the specific model to calculate transient fuel consumption~\cite{akcelik1989efficiency}
		\begin{equation}
			\label{equ2:Akcelik}
			G_{i}(t)=\alpha+\beta_{1}P_{T}(t)+\left(\beta_{2} m a_{i}(t)^{2} v_{i}(t)\right)_{a_{i}(t)>0},
		\end{equation}
		where $ m $ is the vehicle mass, and the term $ \left(\beta_{2} m a_{i}(t)^{2} v_{i}(t)\right)_{a_{i}(t)>0} $ represents the extra inertial (engine/internal) drag power in vehicle acceleration. $\alpha$ is the idle fuel consumption rate and $P_{T}$ denotes the total power to drive the vehicle, which contains the engine dragging power, moment of inertia, air friction and other energy loss; it can be computed by
		\begin{equation}
			\label{equ2:Akcelik:Power}
			P_{T}(t)=\max \left\{0, d_{1} v_{i}(t)+d_{2} v_{i}(t)^{2}+d_{3} v_{i}(t)^{3}+ma_{i}(t)v_{i}(t)\right\}.
		\end{equation}
		
		As suggested in~\cite{akcelik1989efficiency}, we consider a typical setup for parameter values in the Akcelik's fuel consumption model~\eqref{equ2:Akcelik} and~\eqref{equ2:Akcelik:Power}, as shown in Table~\ref{tab:Akcelik}.
	
	\begin{remark}
		Note that the minimization of fuel consumption is one typical control objective for control of individual vehicles at the intersection, which is known as the \emph{eco-approaching behavior}~\cite{yang2017eco,li2015eco, jiang2017eco}. However, existing research mostly focused on the behaviors of the CAVs themselves; such consideration might limit the potential of CAVs in improving traffic performance, especially in mixed traffic flow where HDVs also exist. One of the major distinctions in our optimal control framework from previous results~\cite{yang2017eco,li2015eco, jiang2017eco} lies in the explicit consideration of the fuel consumption of both CAVs and HDVs. This framework allows one to improve the fuel economy of the entire mixed traffic intersection via direct control of only CAVs.
	\end{remark}
	
	\begin{table}[t]
		\centering
		\begin{tabular}[c]{cccccccc}
			\toprule
			\multirow{2}{0.09\textwidth}{\centering\textbf{Parameter}} & 
			\multirow{2}{0.09\textwidth}{\centering\textbf{$\bm{\alpha}$\\($\bm{\mathrm{ml/s}}$)}} & 
			\multirow{2}{0.09\textwidth}{\centering\textbf{$\bm{\beta_{1}}$}\\} & 
			\multirow{2}{0.09\textwidth}{\centering\textbf{$\bm{\beta_{2}}$}\\} & 
			\multirow{2}{0.09\textwidth}{\centering\textbf{$ \bm{m} $\\($\bm{\mathrm{kg}}$)}} & 
			\multirow{2}{0.09\textwidth}{\centering\textbf{$\bm{d_{1}}$}\\} & 
			\multirow{2}{0.09\textwidth}{\centering\textbf{$\bm{d_{2}}$}\\} & 
			\multirow{2}{0.09\textwidth}{\centering\textbf{$\bm{d_{3}} $}\\}\\
			& & & & & & & \\
			\midrule
			Value & 0.666 &  0.072  &  0.0344 & 1680 & 0.269 & 0.0171 & 0.000672 \\
			\bottomrule
		\end{tabular}
		\caption{Parameter setup in the fuel consumption model~\eqref{equ2:Akcelik} and~\eqref{equ2:Akcelik:Power}}
		\label{tab:Akcelik}
	\end{table}

	
	\begin{table}[t]
		\centering
		\begin{tabular}[c]{ccccccc}
			\toprule
			\multirow{2}{0.11\textwidth}{\centering\textbf{Parameter}} & 
			\multirow{2}{0.11\textwidth}{\centering\textbf{$\bm{\kappa}$\\($\bm{\mathrm{s^{-1}}}$)}} & 
			\multirow{2}{0.11\textwidth}{\centering\textbf{$\bm{V_1}$\\($\bm{\mathrm{m/s}}$)}} & 
			\multirow{2}{0.11\textwidth}{\centering\textbf{$\bm{V_2}$\\($\bm{\mathrm{m/s}}$)}} & 
			\multirow{2}{0.11\textwidth}{\centering\textbf{$ \bm{C_1} $\\($\bm{\mathrm{m^{-1}}}$)}} & 
			\multirow{2}{0.11\textwidth}{\centering\textbf{$\bm{C_2}$}\\} & 
			\multirow{2}{0.11\textwidth}{\centering\textbf{$\bm{L_\mathrm{veh}}$\\($\bm{\mathrm{m}}$)}}\\
			& & & & & & \\
			\midrule
			Value & 0.85 &  6.75  &  7.91  & 0.13 & 1.57 & 5 \\
			\bottomrule
		\end{tabular}
		\caption{Parameter setup in the OVM model~\eqref{equ2:OVM_a} and~\eqref{equ2:OVM_vopt}}
		\label{tab:OVM-parameter}
	\end{table}
	
	
	\subsubsection{Terminal Velocity}
	\label{Sec:Terminal Speed}
	
	\begin{figure*}[t]
		\centering
		\subfigure[\label{fig:DistanceNumber:a}] {\includegraphics[width=0.4\linewidth]{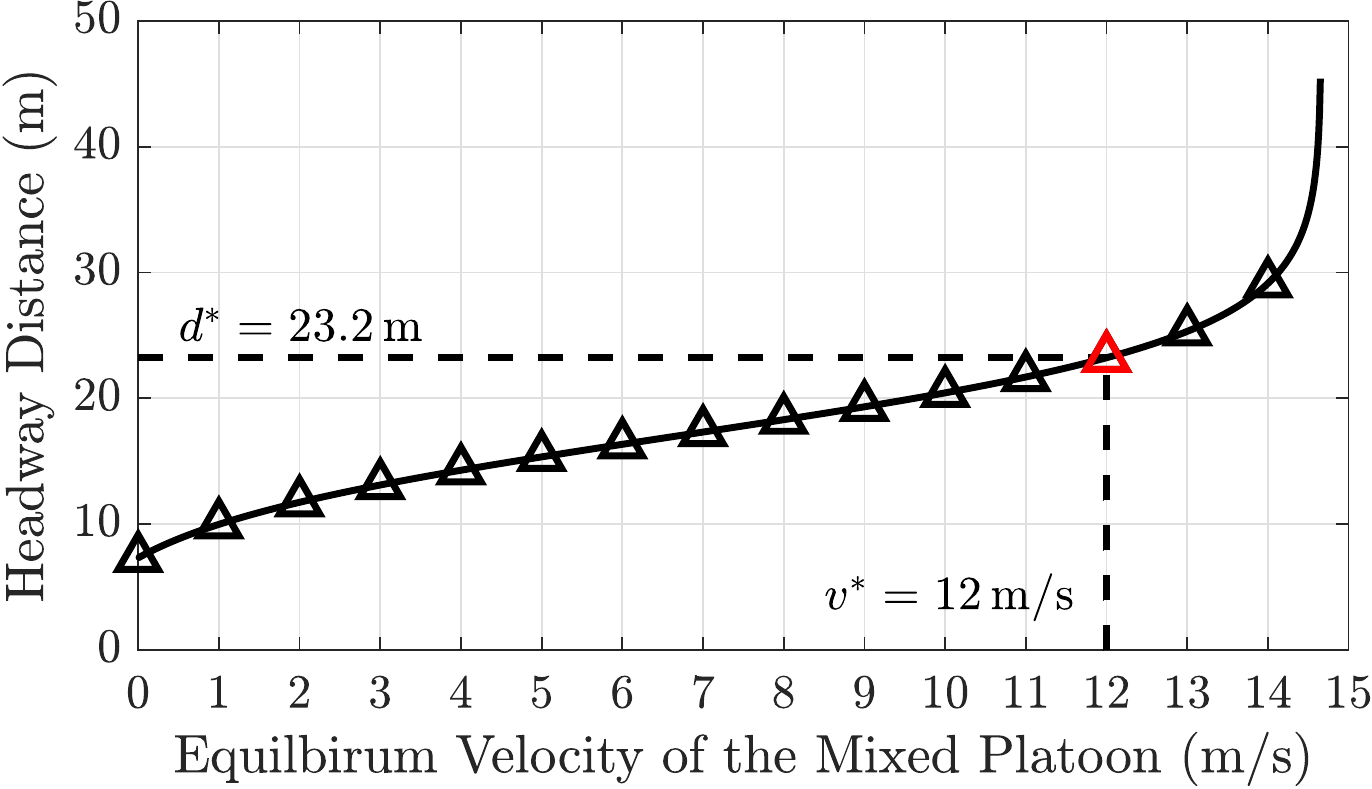}} 
		\hfil
		\subfigure[\label{fig:DistanceNumber:b}] {\includegraphics[width=0.4\linewidth]{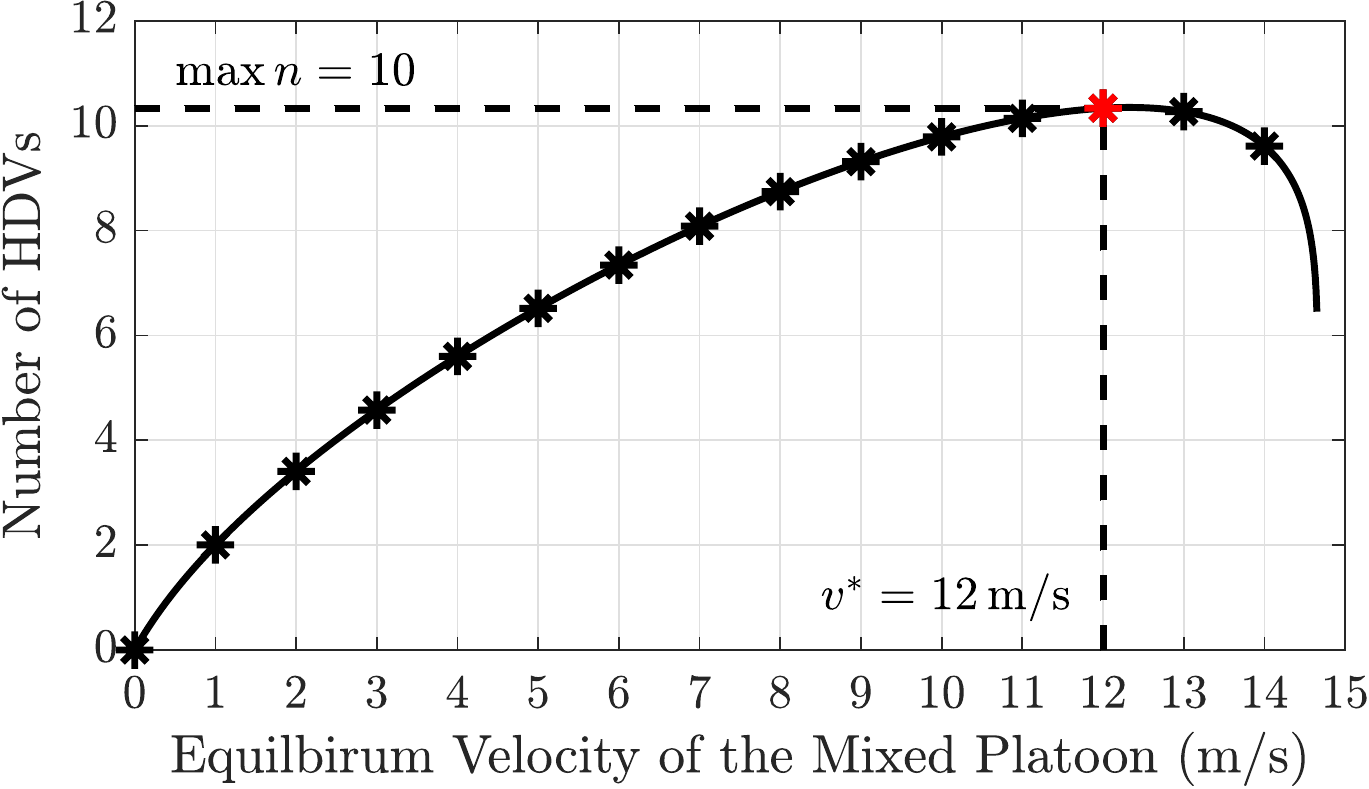}}
		\caption{{Relationship between equilibrium velocity of the mixed platoon (\ie, terminal platoon velocity) and passing number $n$ of HDVs. OVM is taken as the car-following model, whose parameters are shown in Table~\ref{tab:OVM-parameter}. With the increase of equilibrium velocity, the equilibrium headway distance of HDVs increases, while the passing number of HDVs first rises and then drops. Hence, for constant traffic SPAT, there exists a maximum value of equilibrium velocity that maximizes the passing number of HDVs, highlighted as the red point.} \label{fig:DistanceNumber}}
	\end{figure*}
	
	We proceed to discuss how to design the desired equilibrium velocity $v^*$, which also represents the terminal velocity in the terminal cost function~\eqref{equ2:ternimal_cost_fuction}. Existing research mostly focused on the control of the CAV alone, and thus they typically set the terminal velocity of the CAV as the highest limited velocity in order to improve traffic efficiency at the intersection; see, \eg,~\cite{asadi2011predictive,jiang2017eco}. Considering that there might exist other HDVs at the intersection, we reveal that this setup in previous works might not be the optimal choice for the entire mixed intersection. 
	
	When designing the terminal velocity, we aim at maximizing the number of vehicles that can pass the intersection in an equilibrium state during a constant green phase time $ T_{\mathrm{Green}} $. 
	Take one ``$ 1+n$'' mixed platoon for example, \ie, there is one leading CAV and $ n $ following HDVs in the platoon. From the equilibrium equation~\eqref{equ:CFmodel:Stable} in the HDVs' car-following model, it can be inferred that the equilibrium headway distance $ d^* $ relies on the equilibrium velocity $ v^* $. For constant green light phase time, our optimization goal is to maximize the number $ n $ of the following HDVs in $ T_{\mathrm{Green}} $. Accordingly, we have the following result.
		\begin{definition}[Optimal velocity $ v^{*} $]
			\label{Th:TerminalVelocity}
			Consider the ``$ 1+n $'' mixed platoon system consisting of one leading CAV and $n$ following HDVs given by~\eqref{equ:MixedPlatoon}. The optimal equilibrium velocity $ v^{*} $, \ie, the optimal target velocity, is defined as the equilibrium velocity of the mixed platoon which maximizes the passing number $ n $ of the following HDVs during a constant green phase time $ T_{\mathrm{Green}} $. It can be obtained by solving the following optimization problem:
			\begin{equation}
				\label{equ2:T_Green:N}
				\begin{aligned}
					&\arg\max_{v^*} &n = \frac{v^* T_{\mathrm{Green}}}{d^*}, \\
					&\mathrm{subject\; to: } &F(d^*,0,v^*)=0.
				\end{aligned}
			\end{equation}
		\end{definition}
	
	{Recall that we utilize the OVM model~\eqref{equ2:OVM_a} to derive the explicit expression of HDVs' car-following model $F(\cdot)$, with the specific value of the linearized coefficients in~\eqref{equ:ErrorSpeedDiff} calculated by~\eqref{equ2:OVM_vopt_alphai}. Motivated by~\cite{helbing1998generalized}, we consider a typical parameter setup of the OVM model as shown in Table~\ref{tab:OVM-parameter}. Note that both open-loop stability and controllability hold in this parameter setup.} 
	
	{The leading CAV is expected to reach the stopping line when the traffic light turns green.} Accordingly, if the mixed platoon is in equilibrium state, the leading CAV and following HDVs run at the same velocity, \ie, $F(\cdot)= \dot{v}_{i} = 0$ holds for the following HDVs. Thus, we obtain the relationship between velocity and headway distance in equilibrium state based on the OVM model as follows. Substituting $F(\cdot)= \dot{v}_{i} = 0$ into equation~\eqref{equ2:OVM_vopt} yields
	\begin{equation}
		\label{equ2:Steady_distance}
		d_i = \frac{1}{C_1}\left(\arctan\left(\frac{v_i - V_1}{V_2}\right) + C_2\right)  + L_\mathrm{veh}.
	\end{equation}
	
	{In Fig.~\ref{fig:DistanceNumber}, it can be observed that the equilibrium headway distance of HDVs is typically a monotonically increasing function with respect to equilibrium velocity in the equilibrium state. There exists a maximum passing number of HDVs corresponding to the optimal equilibrium velocity $ v^{*} $ and equilibrium headway distance $ d^{*} $. The optimal terminal velocity $ v^{*} $ can be obtained by solving \eqref{equ2:T_Green:N}.}
	
	{
	\begin{remark}
		We make further analysis on the optimal velocity $v^*$. If $ F(d^*,0,v^*)=0 $ leads to an explicit expression of $ d^* = d(v^*) $, solving~\eqref{equ2:T_Green:N} yields 
		\begin{equation}
		\frac{\partial}{\partial{v^*}}\left(\frac{v^*T_{\mathrm{Green}}}{d(v^*)}\right) = 0, 
		\end{equation}
		which leads to
		\begin{equation}
		\label{equ:vstar_Tgreen}
		d(v^*) - d'(v^*)v^* = 0. 
		\end{equation}
		From \eqref{equ:vstar_Tgreen} it can be observed that the optimal velocity $ v^* $ is only related to the car-following model and its equilibrium equation $ F(d^*,0,v^*)=0 $. It has no relationship from the signal phasing time $ T_{\mathrm{Green}} $. By contrast, the maximum passing number $ n $ depends on the value of both $ T_{\mathrm{Green}} $ and $ v^* $. 
	\end{remark}
}

	\begin{remark}
		Our result is consistent with the typical observations from the perspective of macroscopic traffic theory. Denote $\rho(x,t),v(x,t),q(x,t)$ as the traffic density, traffic flow velocity and traffic flow volume at position $x$ and time $t$, respectively. {The fundamental Lighthill-Whitham-Richard model~\cite{lighthill1955kinematic} is commonly employed to depict the relationship among them, which is shown in~\eqref{equ2:LWR} and~\eqref{equ2:LWR_flow}
			\begin{equation}
				\label{equ2:LWR}
				\frac{\partial\rho(x,t)}{{\partial}t} + \frac{{\partial}q(x,t)}{{\partial}x} = 0,
			\end{equation}
			
			\begin{equation}
				\label{equ2:LWR_flow}
				q(x,t) = Q(\rho(x,t)),
		\end{equation}}
		where $ Q $ is usually a convex and non-monotonic function---typical results reveal that the traffic flow volume $q$ usually grows up first and then drops down as the increase of the traffic density $\rho$. Similarly, our result suggest that the optimal target velocity for the mixed platoon is not "the higher, the better".
	\end{remark}

	\subsubsection{Constraints}
	\label{Sec:Constraints}
	For practical implementation of the obtained controller of the CAV, there also exist several constraints that need to be taken into consideration, including process constraints and terminal constraints. Regarding process constraints, first is the safety constraint, which means that each vehicle in the mixed platoon should keep a safe distance $d_{\mathrm{safe}}$ from the preceding vehicle. 
	\begin{equation}
		\label{equ2:Constraint_Distance}
		x_{i}(t) - x_{i-1}(t) - L_{\mathrm{veh}} \ge d_{\mathrm{safe}},\, \mathrm{for}\;t_0 \leq t \leq t_{\mathrm{f}},\,i=1,2,\ldots,n.
	\end{equation}
	
	Second is the practical constraint of the value of the velocity and acceleration of each vehicle in the mixed platoon. Denote $v_{\mathrm{max}}$ as the maximum velocity, $a_{\mathrm{min}}$ and $a_{\mathrm{max}}$ as the minimum and maximum acceleration, respectively. Then, it should be satisfied that
	\begin{align}
		0 \le v_{i}(t) \le v_{\max},\, &\mathrm{for}\;t_0 \leq t \leq t_{\mathrm{f}},\,i=0,1,2,\ldots,n ; \label{equ2:Constraint_Speed}\\
		a_{\min} \le  {a}_{i}(t) {\le}a_{\max},\, &\mathrm{for}\;t_0 \leq t \leq t_{\mathrm{f}},\,i=0,1,2,\ldots,n. \label{equ2:Constraint_Acceleration}
	\end{align}

	{For the terminal constraint, we mainly focus on the terminal position of the CAV. Recall that the deviation of the terminal position $ x_{0}(t_{\mathrm{f}}) $ of the CAV from the target position $x_{\mathrm{tar}}$ has been penalized in the terminal cost function~\eqref{equ2:ternimal_cost_fuction}. Here, we further add an inequality constraint to require the CAV to neither pass the stopping line nor leave a large spacing away from it, given as follows
		\begin{equation}
			\label{equ2:Constraint_Ternimal_Position_Relax}
			0 \le x_{0}(t_{\mathrm{f}}) \le  x_{0}^{\max}(t_{\mathrm{f}}),
		\end{equation}
		where $x_{0}^{\max}$ denotes the maximum tolerance spacing of the CAV from the stopping line. Note that most single CAV control algorithms (\eg,~\cite{asadi2011predictive}) would not incorporate the penalty for the terminal position $ x_{0}(t_{\mathrm{f}}) $ in the terminal cost. Instead, they typically consider a hard constraint to require $ x_{0}(t_{\mathrm{f}}) = 0 $. These algorithms are mostly suitable for fully-autonomous scenarios. In our work for mixed traffic intersections, by contrast, the control of the ``$1+n$'' mixed platoon aims at passing the intersection with an optimal velocity $ v^{*} $ and meanwhile reducing fuel consumption. If we predetermine a fixed setup for both terminal time $ t_\mathrm{f} $ and terminal position $ x_{0}(t_{\mathrm{f}}) $ similar to those single CAV control algorithms, the feasible region for the control input of the leading CAV could be greatly limited, which might jeopardize the optimal performance for the entire mixed platoon system.}
	
	\begin{remark}
		Note that the process constraint~\eqref{equ2:Constraint_Distance} only focuses on the longitudinal position inside a ``$1+ n $'' mixed platoon. The collisions between two adjacent ``$1+ n $'' mixed platoons are not considered in the constraints in the optimal control formulation. To address the safety constraint between different mixed platoons, we introduce an event-triggered algorithm, which is presented in Section~\ref{Section:Algorithm}.
	\end{remark}
	
	\subsubsection{Optimal Control Formulation}
	Lumping the aforementioned design of the cost function and the constraints, the overall optimal control problem can be formulated as follows
	\begin{equation}
		\newcommand{\ud}{\mathrm{d}}
		\label{equ:OptimalControl}
		\begin{aligned}
			&\mathop{\arg\max}_{u(t)}\; J = \varphi({X}(t_\mathrm{f})) + \int_{t_0}^{t_\mathrm{f}}L({X}(t),{u}(t))\ud t, \\
			&\mathrm{subject\; to: }\; \eqref{equ2:OVM_a}, \eqref{equ2:OVM_vopt}, \eqref{equ2:Constraint_Distance}, \eqref{equ2:Constraint_Speed}, \eqref{equ2:Constraint_Acceleration},\eqref{equ2:Constraint_Ternimal_Position_Relax}, \\
			&\mathrm{given: }\; {X}(t_\mathrm{0}).
		\end{aligned}
	\end{equation}
	
	\begin{table}[t]
		\centering
		\begin{tabular}[c]{cc}
			\toprule
			{\textbf{Parameter}} &  {\textbf{Value}} \\
			\midrule
			setup.nlp.solver &  SNOPT \\
			setup.nlp.snoptoptions.tolerance &  $ 2 \times $10$^{-3}$ \\
			setup.scales.method &  automatic-bounds  \\
			setup.derivatives.derivativelevel  & second  \\
			setup.mesh.tolerance &  $ 10^{-2}$ \\
			setup.mesh.iteration &  $ 8 $ \\
			setup.mesh.method &  hp1 \\
			setup.method &  RPMintegration \\
			\bottomrule
		\end{tabular}
		\caption{GPOPS Parameters}
		\label{tab3:OptimalControlSolve}
	\end{table}
	
	Before solving Problem~\eqref{equ:OptimalControl}, the optimal terminal velocity $ v^{*} $ needs to be calculated first by solving Problem~\eqref{equ2:T_Green:N}. In addition, the terminal time $ t_\mathrm{f} $ needs to be pre-determined, which is discussed in the proposed algorithm in Section~\ref{Section:Algorithm}. To solve Problem~\eqref{equ:OptimalControl} numerically, which is a high-order nonlinear optimal control problem, we transform it into nonlinear programming (NLP) problem by employing the pseudo-spectral method~\cite{elnagar1995pseudospectral}. Several practical packages can be directly utilized to address this problem; see, \eg, the GPOPS (General-Purpose Optimal Control Software) toolbox~\cite{patterson2014gpops} with parameter setups shown in Table~\ref{tab3:OptimalControlSolve}. {Note that the optimal control problem can be formulated by utilizing any car-following model that satisfies~\eqref{equ:CFmodel}. In this paper we focus on OVM~\eqref{equ2:OVM_a}~\eqref{equ2:OVM_vopt} considering its balance between model fidelity and computational tractability. Indeed, the adopted pseudo-spectral method has its limitation when more complex nonlinear car-following models (\eg, IDM) are used. It is a significant future direction to seek for more efficient numerical methods to solve this high-order nonlinear optimal control problem.}


	\subsection{Algorithm Design}
	\label{Section:Algorithm}
	In this section, we firstly introduce the benchmark algorithm for the control of CAVs at signalized intersections. Note that only the constraints in the mixed platoon are considered in Section~\ref{Sec:Constraints}. To avoid the potential collision between mixed platoons, we also present our event-triggered algorithm design for practical implementation of the ``$1+n$'' mixed platoon based method.
	
	\begin{figure*}[!t]
		\centering
		\includegraphics[width=\linewidth]{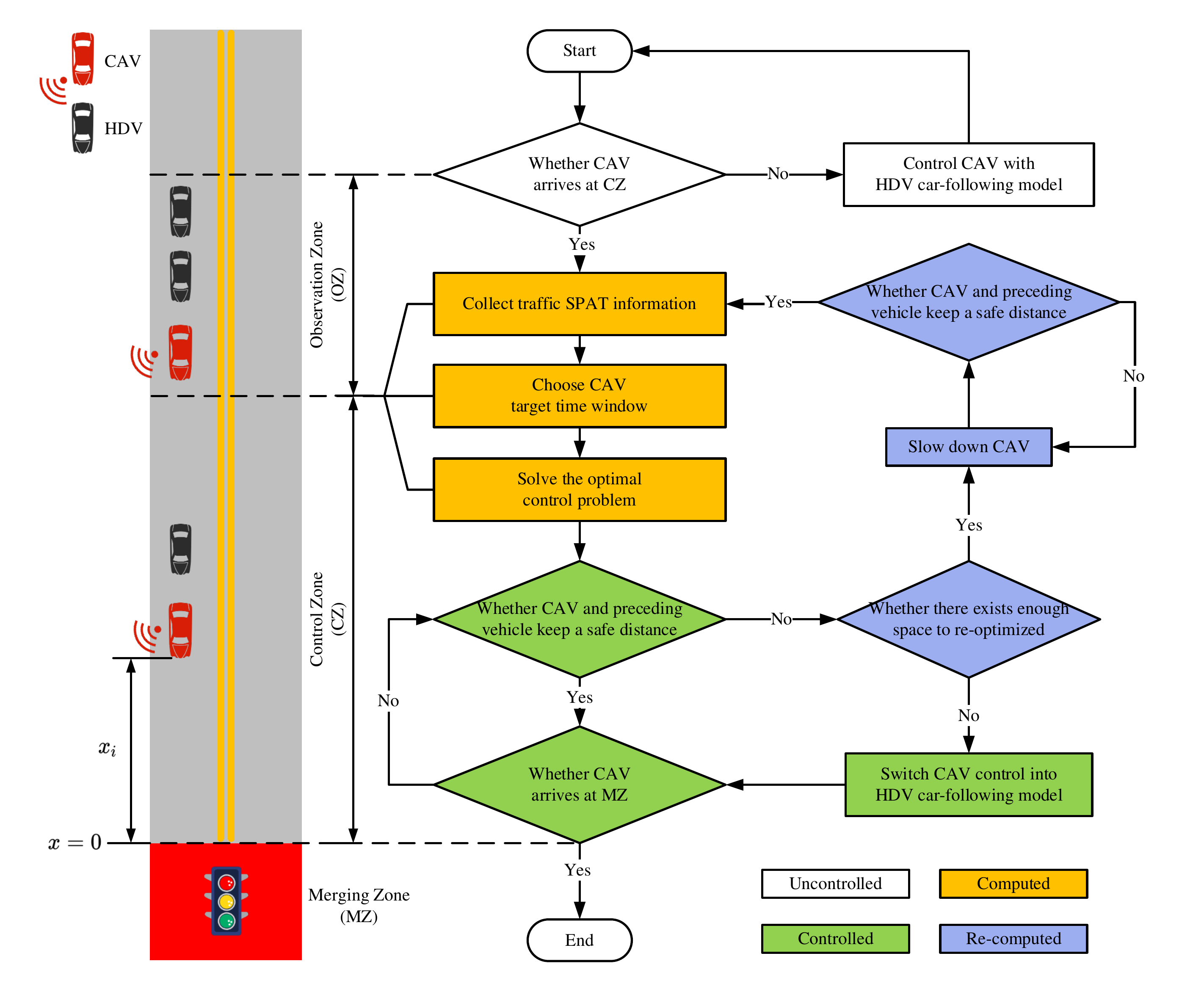}
		\caption{Mixed platoon algorithm. The left figure illustrates one arm of the intersection, which is separated into three zones corresponding to Section~\ref{section:ScenarioSetup}. {Note that the stopping line is set as the $ x=0 $ position and vehicle position $ x_i $ is calculated correspondingly.} The right figure illustrates the mixed platoon algorithm. CAV has four states. Before the CAV arrives at CZ, it is in \emph{uncontrolled} state. CAV becomes \emph{computed} state when it arrives at the boundary of OZ and CZ, where the optimization is performed. After entering CZ, CAV is in \emph{control} state and carries out the optimized velocity trajectory, unless CAV might collide with preceding vehicle thus it enters \emph{re-computed} state.}
		\label{fig:Algorithm}
	\end{figure*}
	
	\subsubsection{Benchmark Algorithm}
	\label{Sec:BenchmarkAlgorithm}
	The basic benchmark algorithm is the predictive cruise control (PCC) algorithm proposed by Asadi \etal~\cite{asadi2011predictive}. {After obtaining traffic SPAT by V2I technology in advance, a practical method was proposed to choose the target green phase window, shown as follows 
		\begin{equation}
			\label{equ3:PCC+}
			[v_{\mathrm{low}},v_{\mathrm{high}}] = \left[\frac{D_{k}}{r_{j} - t}, \frac{D_{k}}{g_{j} - t}\right] \cap\left[v_{\min}, v_{\max}\right],
		\end{equation}
		where $ D_{k} $ is the distance from the CAV $ k $ to the stopping line; $ t $ is the current time; $ r_{j} $ is the start time of the next $ j $th red phase; $ g_{j} $ is that of the next $ j $th green phase; $ v_{\min} $ and $ v_{\max} $ are the CAV velocity limitations. Recall that the green light time in Definition~\ref{Th:TerminalVelocity} can be obtained as $ T_{\mathrm{Green}} = r_{j} - g_{j} $ where $ T_{\mathrm{Green}} > 0 $. The nonempty intersection $ [v_{\mathrm{low}},v_{\mathrm{high}}] $ shown in~\eqref{equ3:PCC+} indicates the feasible velocity window which allows the CAV to pass without idling. Then, the CAV obtains the target velocity $ v_{\mathrm{target}} = v_{\mathrm{high}} $, calculates the corresponding terminal time $ t_{\mathrm{f}} = \frac{D_{k}}{v_{\mathrm{target}}}$, and schedules an optimum velocity trajectory through model predictive control (MPC)~\cite{maciejowski2002predictive}.
	}
	
	However, PCC cannot consider preceding HDVs (especially the queuing ones) into optimization. Instead of PCC, we consider an improvement inspired by Yang's queuing length adjustment method~\cite{yang2017eco}. The distance from the CAV $ k $ to the stopping line $ D_{k} $ is optimized by
	
	\begin{equation}
		\label{equ3:PCC+:distance}
		D_{k}^{*} = 
		\begin{aligned} 
			&\frac{v_k}{v_{k}+v_{\mathrm{AC}}}\left[D_{k}+v_{\mathrm{AC}}\left(r_{j}-t\right)\right], \\
		\end{aligned}
	\end{equation}
	where $ v_{k} $ is the current velocity of CAV $ k $, and $ v_{\mathrm{AC}} $ is the velocity of the traffic upstream flow shock wave. In the rest of this paper, this modified PCC algorithm is named as PCC+. 
	
	It can be inferred that even if the queuing length is considered in the optimization by~\eqref{equ3:PCC+:distance}, the inner core of the PCC+ algorithm is still designed for one CAV alone, which makes it a qualified benchmark algorithm to make comparisons with the mixed platoon based algorithm. Note that the similar target time window chosen method~\eqref{equ3:PCC+} is employed in our algorithm.
	
	
	
	\subsubsection{Algorithm Design for Mixed Platoon}
	\label{Section:Algorithm:MixedPlatoon}
	In this paper, we consider the traffic light as constant SPAT. We design a hierarchical mixed platoon based algorithm as shown in Fig.~\ref{fig:Algorithm}. In Section~\ref{Section:OptimalControlMode}, the velocity trajectory planning process has been explained in detail. However, there still remains one problem. In our expectation, the decision-making and trajectory planning processes are made when the CAV arrives at the border between OZ and CZ. But one-time planning cannot forecast the future trajectories of the preceding vehicles, which means there are still collision risks when the CAV is undertaking the planned velocity trajectory in CZ. To solve this problem, there are basically two solutions. One is called shooting heuristic~\cite{ma2017parsimonious}, which estimates the maximum possible trajectory boundaries of the preceding vehicle. The drawback is that the models' deviation makes it difficult to guarantee prediction accuracy. The other method is receding horizon~\cite{feng2015real}, which allows the CAV to make planning in a constant time period. However, this method brings a huge computation burden for real-time implementation. 
	
	
	In fact, trajectory interference seldom happens when the traffic flow density is not saturated. So distributing extra computation brings very limited benefit. Thus, we design an event-triggered method to guarantee safety with minimum calculation. The CAV state is separated into four states, \emph{uncontrolled}, \emph{computed}, \emph{controlled} and \emph{re-computed}.
	
	
	\begin{enumerate}[(1)]
		\item \emph{uncontrolled}
	\end{enumerate}
	
	{
		As shown in the white boxes in Fig.~\ref{fig:Algorithm}, all the CAVs are in \emph{uncontrolled} state when driving in OZ, whose length is denoted as $ L_{\mathrm{obs}} $. During this period, all the CAVs will directly follow the HDVs' car-following model~\eqref{equ2:OVM_a} to control themselves. Meanwhile, all the vehicles, including HDVs and CAVs, are also allowed to change lanes freely in OZ, as highlighted in the assumptions in Section~\ref{section:ScenarioSetup}.
	}
	\begin{enumerate}[(2)]
		\item \emph{computed}
	\end{enumerate}
	
	
	When the CAV arrives at CZ, its state becomes \emph{computed}, as shown in the orange boxes in Fig.~\ref{fig:Algorithm}. At this time point, the CAV firstly receives the traffic SPAT and chooses its target green phase window. Based on the time window, the optimal velocity trajectory is calculated by optimal control framework as shown in Section~\ref{Section:OptimalControlMode}. 
	
	
	As mentioned before, the target window is also chosen by the method in~\eqref{equ3:PCC+}. Differently, only one CAV is considered in~\cite{asadi2011predictive}, and thus it's straightforward to set the target window as the closest one. In our research, however, there is an optimal terminal platoon velocity as designed in Section~\ref{Sec:Terminal Speed}. If the target velocity is set as the maximum velocity $ v_{\max} $, the optimization on the velocity trajectory becomes meaningless. {Therefore, the limitation on the $ v_{\max} $ \eqref{equ3:PCC+} is slacked to
		\begin{equation}
			\label{equ3:Asadi:2}
			[v_{\mathrm{low}},v_{\mathrm{high}}] = \left[\frac{D_{k}}{r_{j} - t}, \frac{D_{k}}{g_{j} - t}\right] \bigcap\left[v_{\min }, \frac{v_{\max } + v^{*}}{2}\right],
		\end{equation}
		in which $ v^{*} $ means the optimal velocity calculated from~\eqref{equ2:T_Green:N}. Similar to the benchmark algorithm presented in Section~\ref{Sec:BenchmarkAlgorithm}, the target velocity is selected through the nonempty intersection $ [v_{\mathrm{low}},v_{\mathrm{high}}] $ with $ v_{\mathrm{target}} = v_{\mathrm{high}} $. The corresponding terminal time is calculated as $ t_{\mathrm{f}} = \frac{D_{k}}{v_{\mathrm{target}}}$, which is utilized as terminal time in the cost function~\eqref{equ2:cost_fuction} in the optimal control formulation.
	}
	
{
		\begin{remark}
			Recall that a maximum value of the passing number $ n $ of vehicles is obtained when solving~\eqref{equ2:T_Green:N}. In practical traffic flow, however, the number of the following HDVs behind the leading CAV might not be exactly equal to $ n $. In particular, when MPR is high, there might be fewer HDVs than $ n $ behind the CAV, and there might be other CAVs except for the leading CAV in the ``$1+n$'' mixed platoon. In this case, our algorithm lets the other CAVs in the ``$1+n$'' mixed platoon to utilize the numerical expression of the HDV's car-following model (precisely, the OVM model~\eqref{equ2:OVM_a}\eqref{equ2:OVM_vopt}) to determine their control input. In this way, the optimal control formulation~\eqref{equ:OptimalControl} is preserved, which allows the design of the control input of the leading CAV. It would be an interesting future direction to apply cooperative control strategies to multiple CAVs in a mixed platoon when passing an intersection, which might further improve the intersection performance.
		\end{remark}
	}
	
	\begin{enumerate}[(3)]
		\item \emph{controlled}
	\end{enumerate}
	
	When the CAV is in CZ, it is in \emph{controlled} state, as the green boxes showed in Fig.~\ref{fig:Algorithm}. In this state, CAV executes the velocity trajectory planned in \emph{computed} state. If there is a preceding vehicle, the CAV simply checks the safety distance constraint~\eqref{equ2:Constraint_Distance} with the preceding vehicle in each step. If not, the CAV's state is changed into \emph{re-computed}.
	
	\begin{enumerate}[(4)]
		\item \emph{re-computed}
	\end{enumerate}
	
	{
		In \emph{controlled} state, if the leading CAV is too close to the preceding vehicle (the threshold distance is set as $6\, \mathrm{m}$ in the simulation), it will become \emph{re-computed} state, which aims to slow down the CAV to avoid collision with the preceding vehicle. Specifically, if the distance to the stopping line is far enough for the CAV to do another optimization ($ D_i > k_{\mathrm{c}}L_{\mathrm{ctrl}} $, where $ L_{\mathrm{ctrl}} $ is the length of CZ and $ k_{\mathrm{c}} \in [0,1] $), it will re-compute its new optimal  trajectory by following the same process in the \emph{computed} state. On the other hand, if the CAV is too close to the intersection ($ D_i \leq k_{\mathrm{c}}L_{\mathrm{ctrl}} $), there is little optimization space for the CAV, and thus in this case, the CAV will directly follow the HDVs' car-following model~\eqref{equ2:OVM_a} to pass the intersection in the remaining time. 
		
	}
	
	\begin{table}[!t]
		\centering
		\begin{tabular}[c]{ccc}
			\toprule
			{\textbf{Parameters}} & {\textbf{Symbol}} & {\textbf{Value}} \\
			\midrule
			Simulation step (s) & $ T_{\mathrm{s}} $ & 0.1 \\
			Maximum acceleration (m/s$^{2} $) & $ a_{\max} $ & 3 \\
			Minimum acceleration (m/s$^{2} $) & $ a_{\min} $ & -6 \\
			Terminal position weight & $ \omega_1 $ & $ 10^5 $ \\
			Terminal velocity weight & $ \omega_2 $ & $ 10^4 $ \\
			Maximum velocity (m/s) & $ v_{\max} $ & 15 \\
			Minimum velocity (m/s)& $ v_{\min} $ & 0  \\
			Control zone length (m) & $ L_{\mathrm{ctrl}} $ & 300 \\
			Observation zone length (m) & $ L_{\mathrm{obs}} $ & 500\\
			\bottomrule
		\end{tabular}
		\caption{Parameter setup for the traffic simulation}
		\label{tab3:OptimalControlMode}
	\end{table}
	
	\section{Simulation Results and Discussion}
	\label{Section:SimulationResult}
	In this section, we evaluate the effectiveness of the proposed optimal control framework in Section~\ref{Section:OptimalControlMode} and the corresponding algorithm in Section~\ref{Section:Algorithm} based on large-scale traffic simulations. The nonlinear OVM model~\eqref{equ2:OVM_a} is utilized to model the dynamics of HDVs. Firstly, the simulation environment and evaluation index are introduced. Then, a case study under $50\%$ MPR is presented to analyze the algorithm performance. Finally, two simulation experiments are conducted, considering multiple traffic volumes and MPRs.
	
	
	\subsection{Simulation Environment and Evaluation Index}
	
	The traffic simulation is conducted in SUMO, which is widely used in traffic researches~\cite{lopez2018microscopic}. The simulation is run on Intel Core i7-7700 processor @3.6GHz. Some simulation parameters are shown in Table~\ref{tab3:OptimalControlMode}.
	
	
	%
	
	Denote $ t_{i}^{\mathrm{in}} $ as the time step when vehicle $ i $ enters CZ, and $ t_{i}^{\mathrm{out}} $ as the time step when it exits CZ. In this way, it means that the vehicle spends $ t_{i}^{\mathrm{out}} - t_{i}^{\mathrm{in}} $ time to travel through CZ, while it spends $ {L_{\mathrm{ctrl}}}/{v_{\max}} $ time to travel through CZ in free drive condition. Accordingly, Average Travel Time Delay (ATTD) is chosen to measure the average traffic efficiency for the mixed platoon, as shown in~\eqref{equ3:TravelTimeDelay}. 
	\begin{equation}
		\label{equ3:TravelTimeDelay}
		\mathrm{ATTD} = \frac{1}{n} \sum_{i = 0}^{n} \left(t_{i}^{\mathrm{out}} - t_{i}^{\mathrm{in}}-\frac{L_{\mathrm{ctrl}}}{v_{\max}}\right).
	\end{equation}
	Meanwhile, fuel consumption per 100km is chosen as the evaluation index to measure the fuel economy. {We consider $ 100 $ vehicles in each experiment. Note that at the initial time of the simulation, there are no vehicles in the traffic network. Since our traffic optimization focuses on steady-state traffic flow, we exclude the performance of the firstly-entered part of the vehicles and focus on that of the last $ 50 $ vehicles in the performance evaluation for the entire intersection.}
	

	\subsection{A Case Study under 50\% MPR}
	
	\begin{figure*}[!t]
		\centering
		\subfigure[No control input assigned to CAVs\label{fig:NC_700}] {\includegraphics[width=0.95\linewidth] {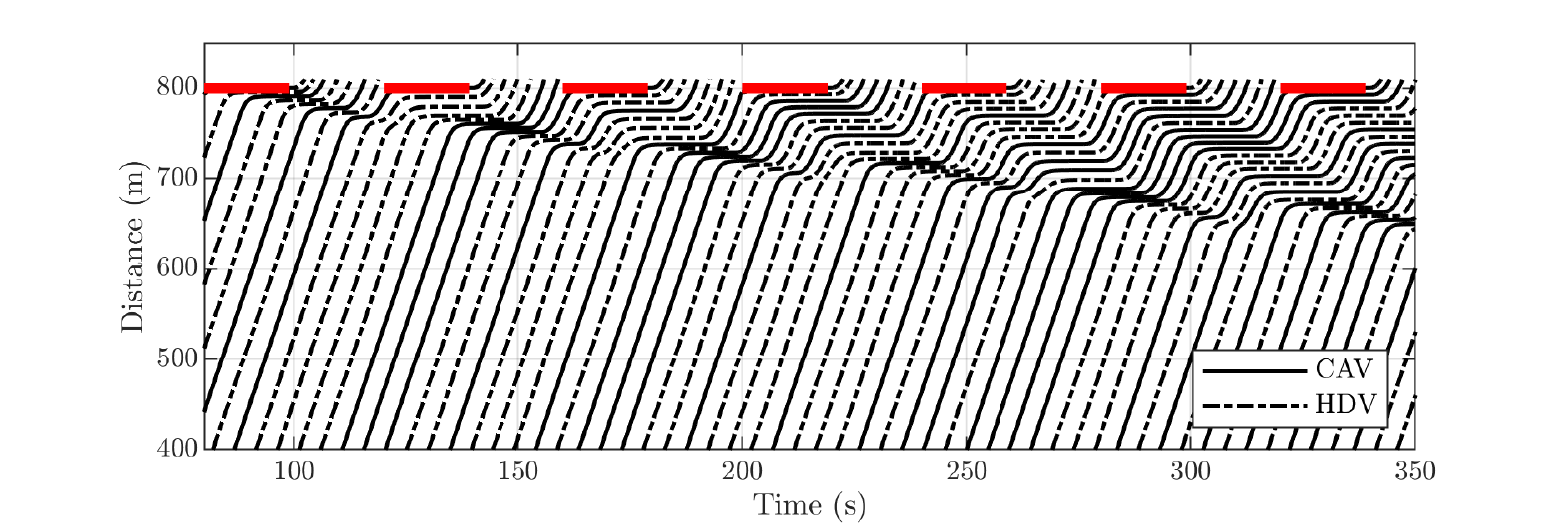}}
		
		\subfigure[PCC+ control input assigned to CAVs\label{fig:OC_700}] {\includegraphics[width=0.95\linewidth] {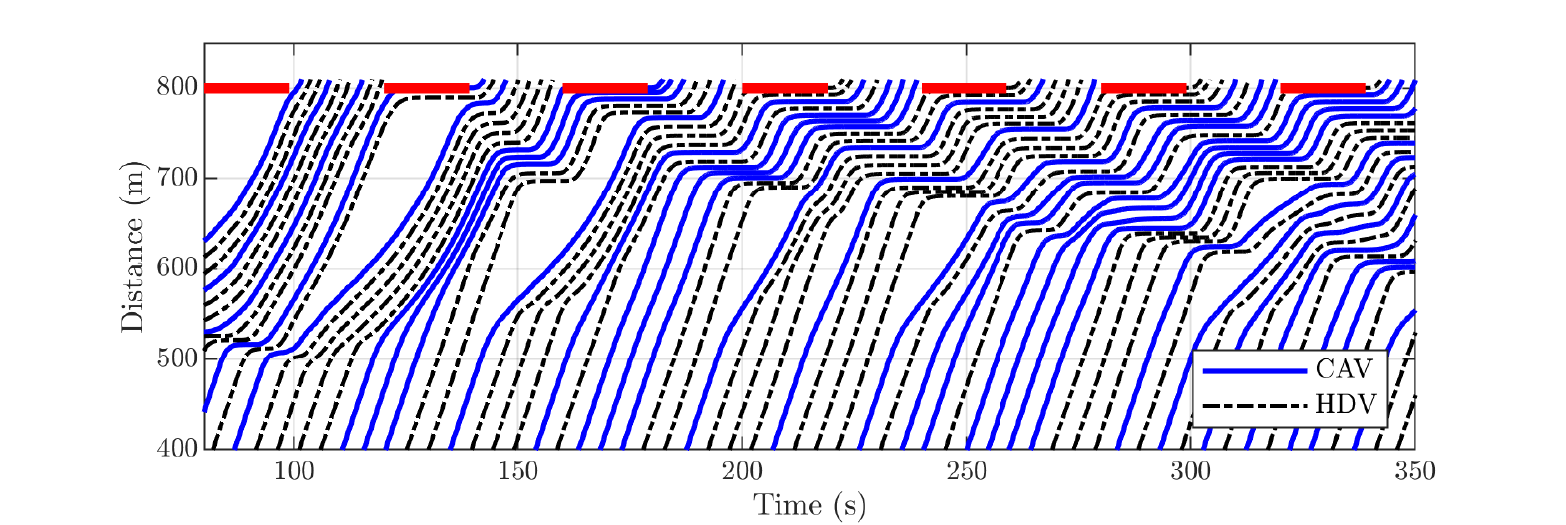}}
		
		\subfigure[Mixed platoon control input assigned to CAVs\label{fig:MC_700}] {\includegraphics[width=0.95\linewidth] {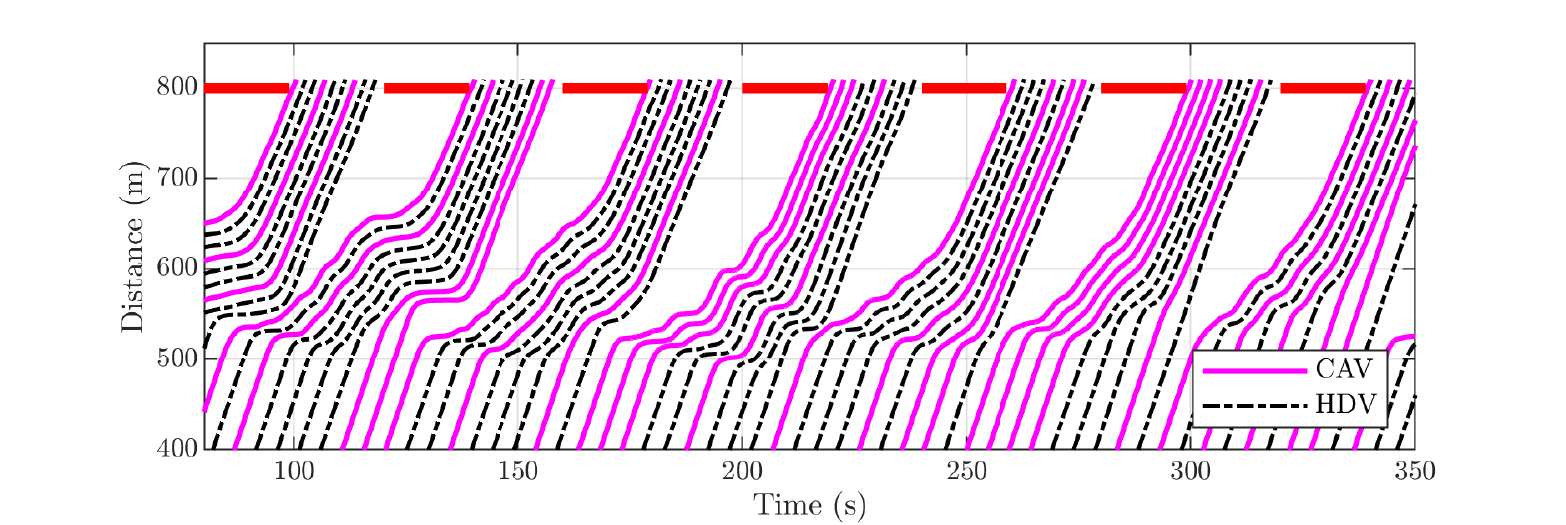}}
		\caption{{Vehicle trajectories of no control applied, PCC+ control and mixed platoon control in $ 750 \, \mathrm{veh/(hour} \cdot \mathrm{lane}) $ (MPR = 50\%). The solid lines represent the CAVs' trajectories, while the dashed dot lines represent HDVs' trajectories. Black solid lines mean there are no optimized control input assigned to CAV. Blue lines are for the PCC+ controlled CAVs while magenta lines are for the proposed mixed platoon control.} \label{fig:ControlCompare}}
	\end{figure*}
	
	We first conduct the simulation in $ 50\% $ MPR of $ 100 $ vehicles, which means there are randomly distributed $ 50 $ CAVs and $ 50 $ HDVs driving into the intersection. Then further comparisons are made between the proposed Mixed Platoon (MP) algorithm and the bench-marking algorithm PCC+ shown in Section~\ref{Section:Algorithm} at different traffic volumes.

	
	%
	
	In Fig.~\ref{fig:ControlCompare} there are three groups of vehicle trajectories in $ 750 \, \mathrm{veh/(hour} \cdot \mathrm{lane}) $, which represent no control input (colored in black), PCC+ control (colored in blue) and proposed mixed platoon control (colored in magenta), respectively. Note that vehicle distributions in three scenarios are exactly the same. It can be observed that in Fig.~\ref{fig:NC_700} all the vehicles are queuing because of the red light, which causes vehicles' idling and fuel waste. In Fig.~\ref{fig:OC_700} the queuing is postponed but still gradually accumulating. On the contrary, with MP applied, the queuing accumulation is constrained in CZ range (300 m). It means that the influence of queuing is limited in one intersection rather than spreading to the upstream intersection, which is important for urban traffic control. In PCC+, the CAV is able to obtain a proper  trajectory for itself. However, the HDVs around it stop the CAV to successfully undertake the optimized trajectory and the queuing for following HDVs behind the CAV is inevitable. What's worse, this might continue to jeopardize the optimization of the next CAV in upstream traffic flow. With HDV considered in the optimization and proper platoon  chosen, the traffic mobility can be further improved as discussed in Section~\ref{Section:OptimalControlMode}. As shown in Fig.~\ref{fig:MC_700}, CAVs are able to lead the following HDVs to pass the intersection with proper velocity in MP.
	
	Note that in MP algorithm, vehicles' idling behavior is avoided as much as possible, and therefore the fuel economy is significantly improved. The specific performance indexes are shown in Table~\ref{Table:Figure5Result}. Moreover, compared with PCC+ control in Fig.~\ref{fig:OC_700}, MP control in Fig.~\ref{fig:MC_700} creates much bigger gaps between trajectory blocks, which creates more optimization space for CAV. Generally speaking, the existence of these gaps means that the CAV algorithm brings lower traffic saturation and higher traffic mobility at the intersections~\cite{feng2018spatiotemporal}. 
	
	{
	To further clarify the performance of MP algorithm in reducing fuel consumption, we present the trajectory of one single vehicle in Fig.~\ref{fig:ControlCompare:No}, indexed as no.~$73$ in Fig.~\ref{fig:ControlCompare}. Note that in No control of Fig.~\ref{fig:NC} or PCC+ control of Fig.~\ref{fig:OC}, the velocity remains $ 0 $ for a long time, which causes a huge amount of fuel consumption when the vehicle stops and the engine idles. The long-time acceleration and deceleration also make the engine less likely to work in high-efficiency working conditions, which worsens the fuel consumption performance. In MP control, on the contrary, the vehicle is always moving forward although the accel/decel cycles are increased moderately (but the oscillation amplitude has been significantly reduced) to arrive at the stopping line with a proper velocity at the proper time. The trajectory is much smoother and the idling behavior is avoided as much as possible. Consequently, the accumulative fuel consumption of this single vehicle under No control or PCC+ reaches $ 0.20\,\mathrm{L} $, while the MP control is only $ 0.12\,\mathrm{L} $.
	}
	
	\begin{table}[t]
		\centering
		\begin{tabular}[c]{cccc}
			\toprule
			& {\textbf{No Control}} &  \textbf{PCC+} & {\textbf{Mixed Platoon}} \\
			\midrule
			{\textbf{ATTD} (s)} & $126.56$ &  $125.26$ & $105.81$ \\
			{\textbf{Fuel Consumption} ($ \mathrm{L}/100\mathrm{km} $)} &  $16.77$ &  $16.66$ & $10.53$ \\
			{\textbf{Accel/Decel Cycles per Vehicle}} & $2.58$ &  $8.64$ & $5.20$ \\
			{\textbf{Ideling Time} (s)} & $40.29$ &  $40.13$ & $2.32$ \\
			\bottomrule
		\end{tabular}
		\caption{{Traffic efficiency and fuel consumption performance comparison of the three algorithm in Fig.~\ref{fig:ControlCompare}.}
		\label{Table:Figure5Result}}
	\end{table}
	
	\begin{figure*}[t]
		\centering		
		\subfigure[No control\label{fig:NC}] {\includegraphics[width=0.32\linewidth] {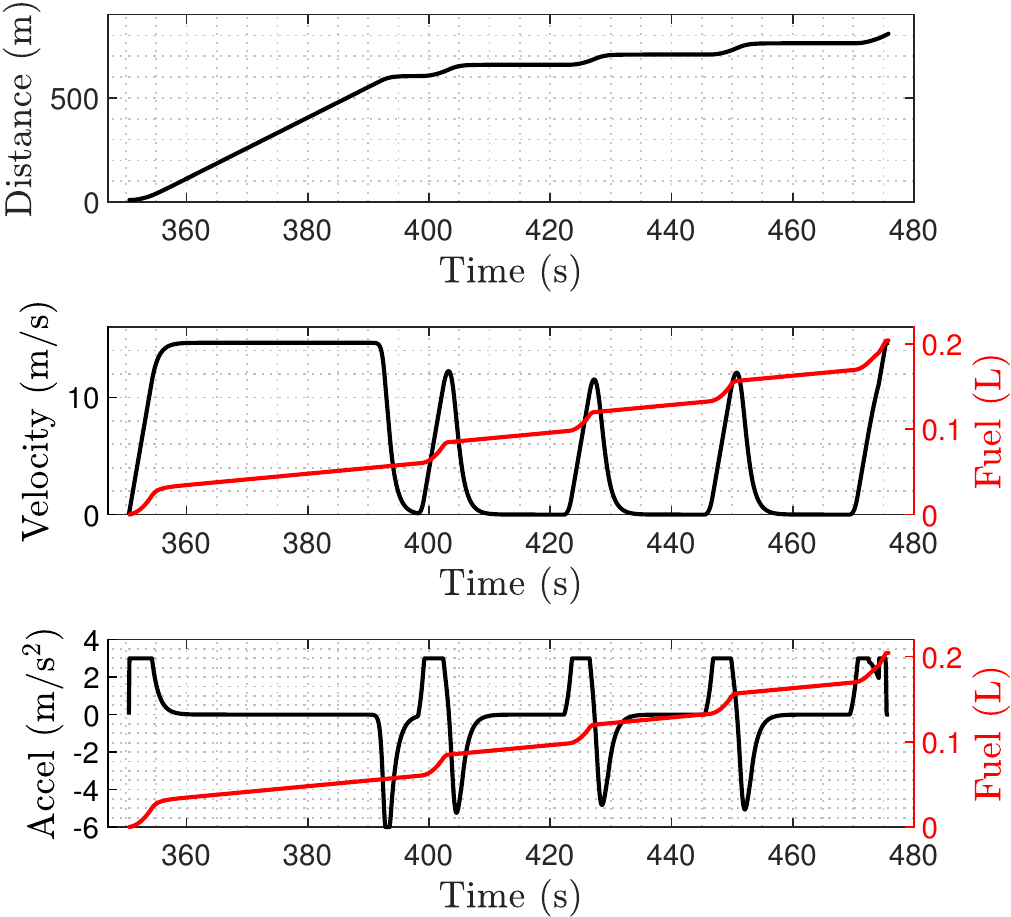}}
		\subfigure[PCC+ control\label{fig:OC}] {\includegraphics[width=0.32\linewidth] {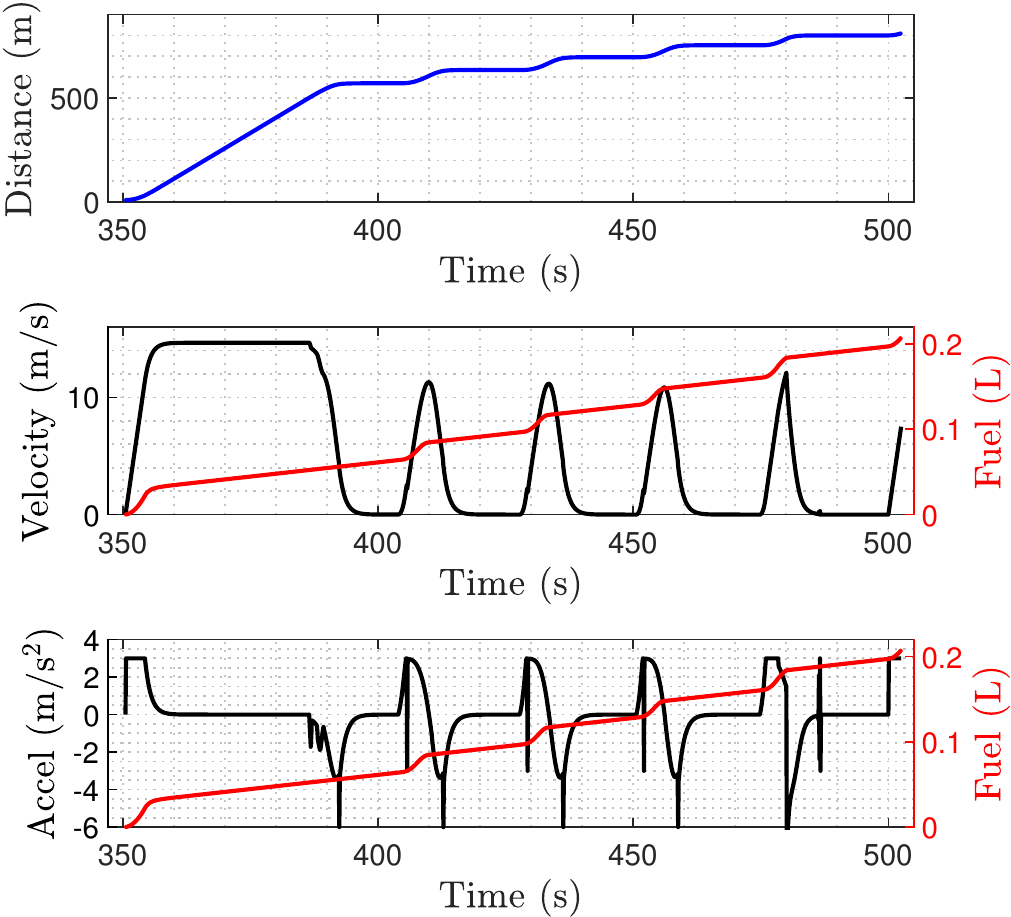}}
		\subfigure[Mixed platoon control\label{fig:MC}] {\includegraphics[width=0.32\linewidth] {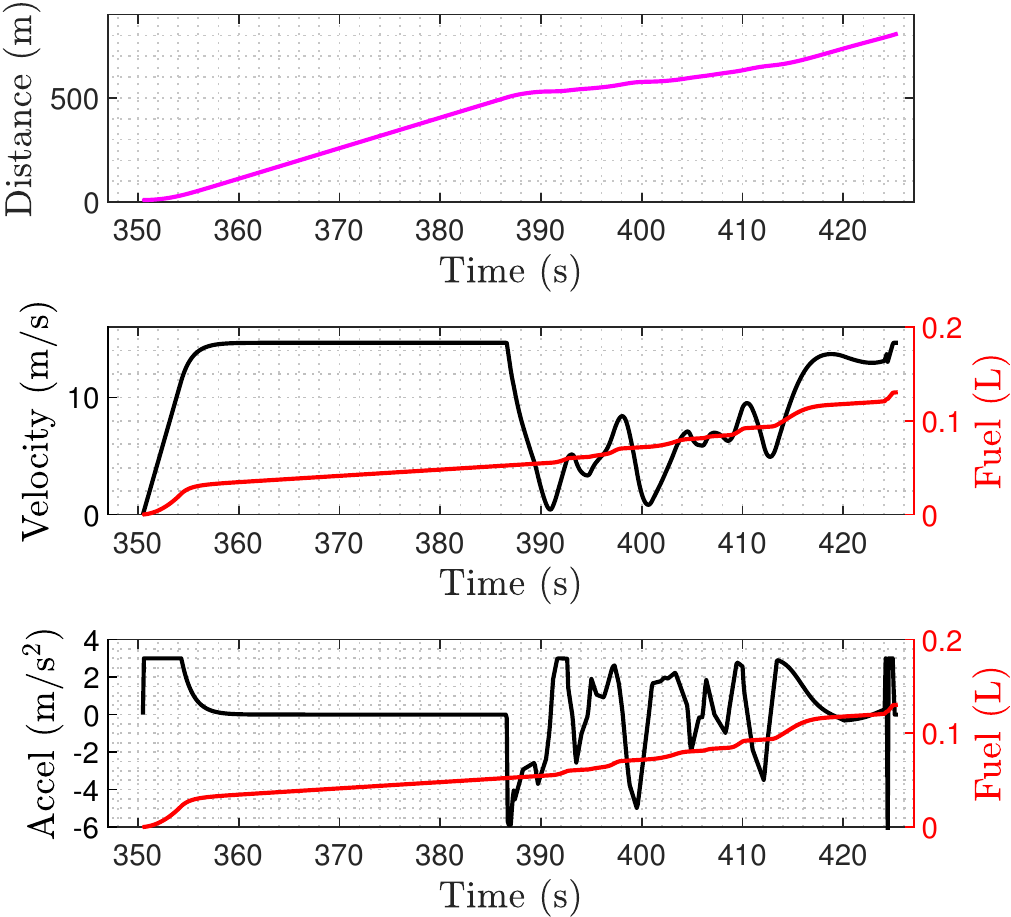}}
		\caption{Trajectories of vehicle no.~$73$ under the three algorithms. Vehicle idling is avoided as much as possible in the MP algorithm compared to the benchmark algorithms. The smoothed velocity trajectory of the MP algorithm helps to improve fuel economy.\label{fig:ControlCompare:No}}
	\end{figure*}

	
	
	
	\subsection{Simulation Results at Different Traffic Volumes}
	\begin{figure*}[t]
		\centering
		\subfigure[Average Travel Time Delay Comparison] {\includegraphics[width=0.4\linewidth]{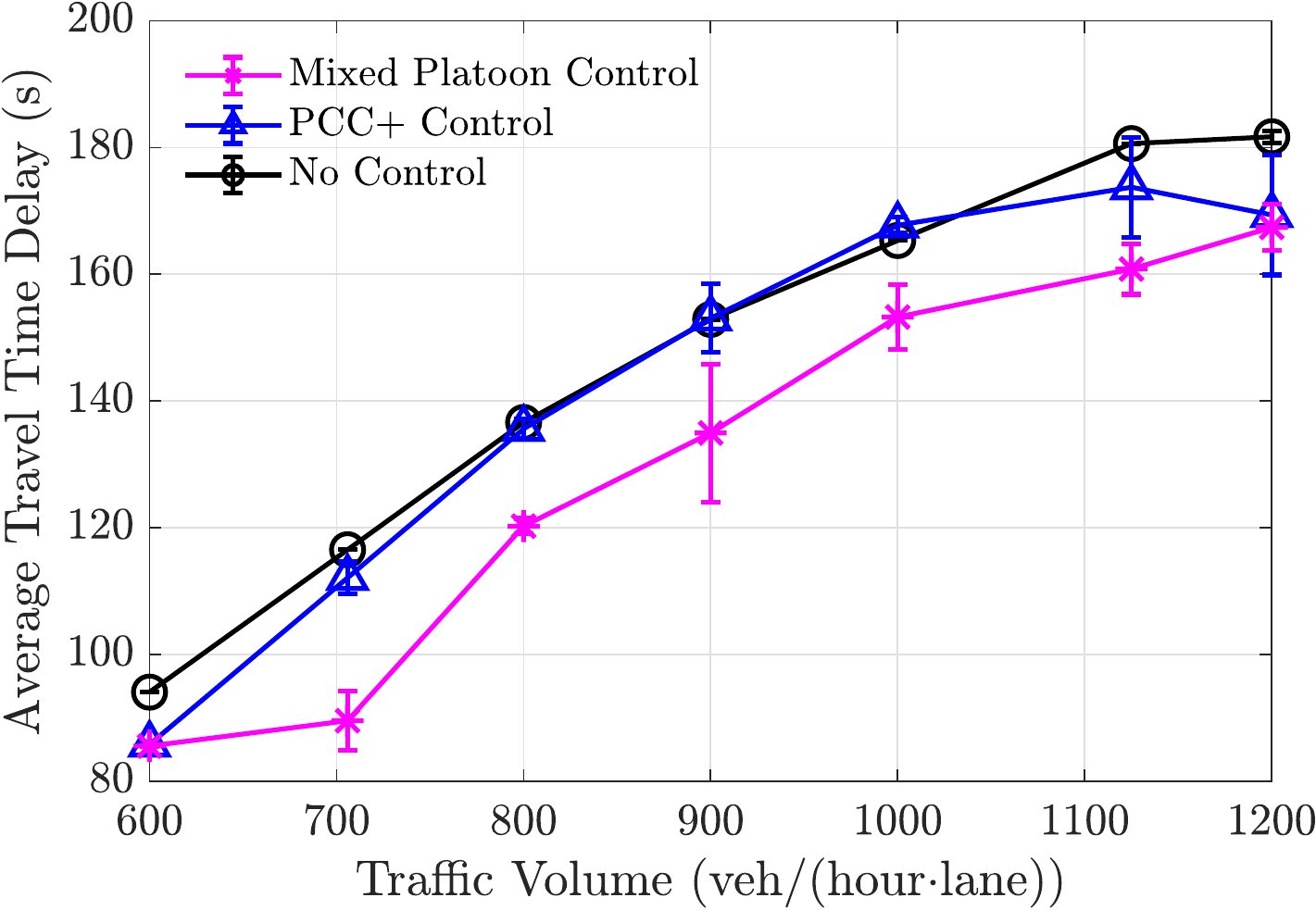}\label{}}
		\hfil
		\subfigure[Fuel Consumption Comparison] {\includegraphics[width=0.4\linewidth]{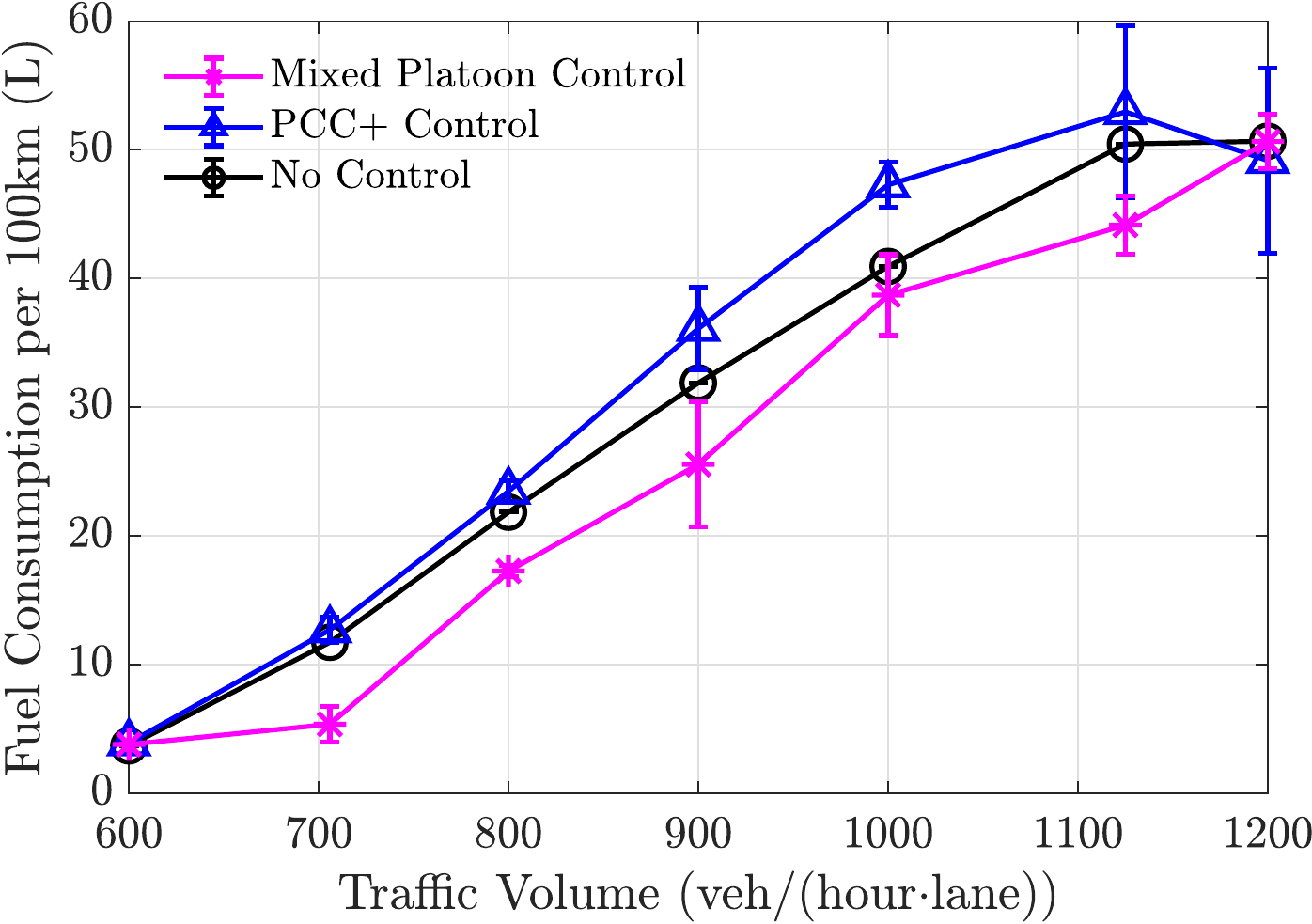}\label{}}
		\caption{{Performance comparison between the three algorithms in different traffic volumes (MPR = 50\%). Standard deviations of the data are also provided as error bars. The black lines means no optimized control input assigned to CAV. Blue lines and magenta lines are for PCC+ and MP algorithm.} \label{fig:ControlCompareSum}}
	\end{figure*}
	Fig.~\ref{fig:ControlCompareSum} compares the performance of the algorithms in various traffic volumes, which is also conducted in $ 50\% $ MPR of $ 100 $ vehicles. {Each working condition is simulated five times and standard deviations are plotted in the error bars as well. Generally speaking, MP has obvious improvement on both traffic efficiency and fuel consumption between $ 600 \mathrm{-} 1200 \, \mathrm{veh/(hour} \cdot \mathrm{lane}) $. When the traffic volume is lower than $ 600 \, \mathrm{veh/(hour} \cdot \mathrm{lane})$, there is no queuing in the intersection no matter which control algorithm is applied to CAV. By contrast, in a much supersaturated traffic volume (\eg, over $ 1200 \, \mathrm{veh/(hour} \cdot \mathrm{lane})$), the optimization space for CAV is rather low considering the high traffic density.
		
	It is clearly observed that PCC+ shows limited improvement in traffic efficiency in the mixed traffic environment. Since HDVs are not considered in the optimization process in PCC+, more velocity fluctuations might happen, which makes the fuel consumption even worse compared with the case of no control applied to CAV.
	The reason why MP outperforms PCC+ is that when the traffic volume increases, the interference between HDVs and CAVs becomes much more frequent. When those HDVs are incorporated in CAV optimization like MP, the CAV can help improve the traffic efficiency and fuel consumption to a higher degree. }
	
	\subsection{Simulation Results at Different Traffic Volumes and MPRs}
	\label{sec:FinalSimulation}
	
	

	In the last simulation, we verify MP algorithm performance in multiple MPRs and traffic volumes. In addition, a stochastic variant for the OVM model~\eqref{equ2:OVM_a} is also under consideration, which is given by
		\begin{equation}
			\dot{v}_{i}=\gamma \kappa[V_{\mathrm{des}}({d}_{i})-v_i],
		\end{equation}
		where $ \gamma \in \left[0.9,1.1\right] $ denotes a stochastic sensitivity gain.

\begin{figure*}[t]
	\centering
	\subfigure[ATTD Improvement\label{fig:ATTD:3D}] {\includegraphics[width=0.45\linewidth]{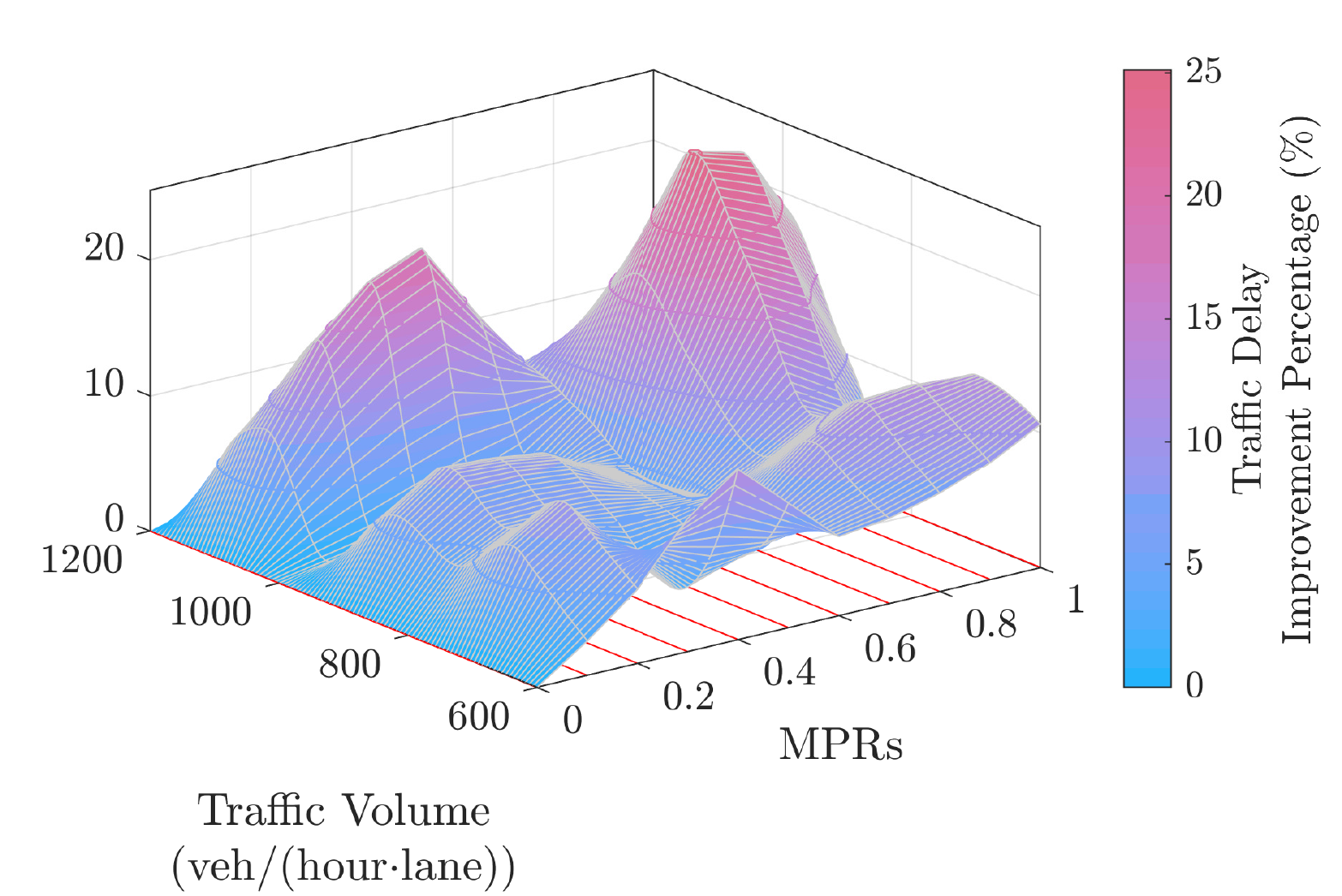}}
	\hfil
	\subfigure[Fuel Consumption Improvement\label{fig:FC:3D}] {\includegraphics[width=0.45\linewidth]{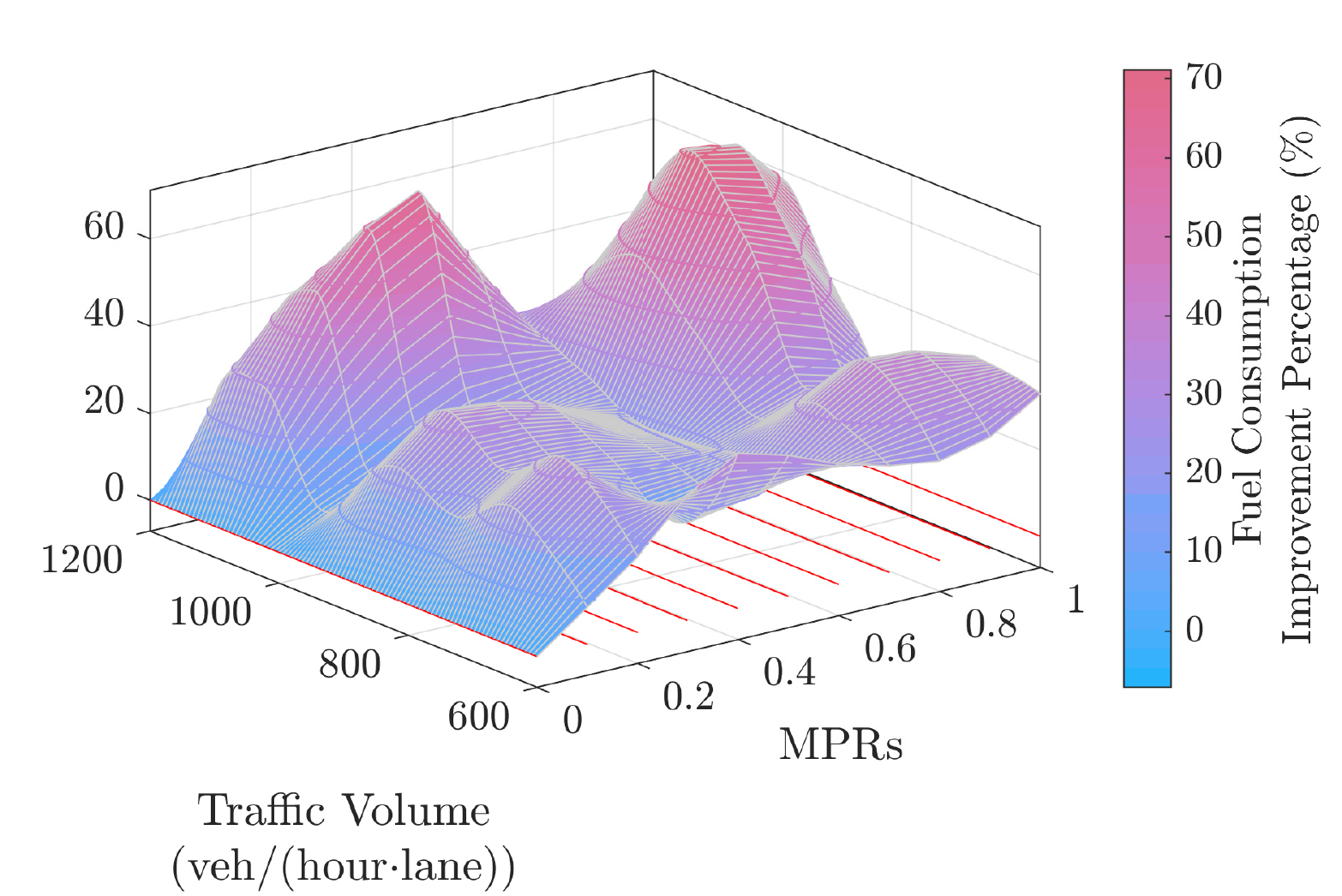}}
	
	\subfigure[ATTD Improvement in 2D\label{fig:ATTD:2D}] {\includegraphics[width=0.45\linewidth]{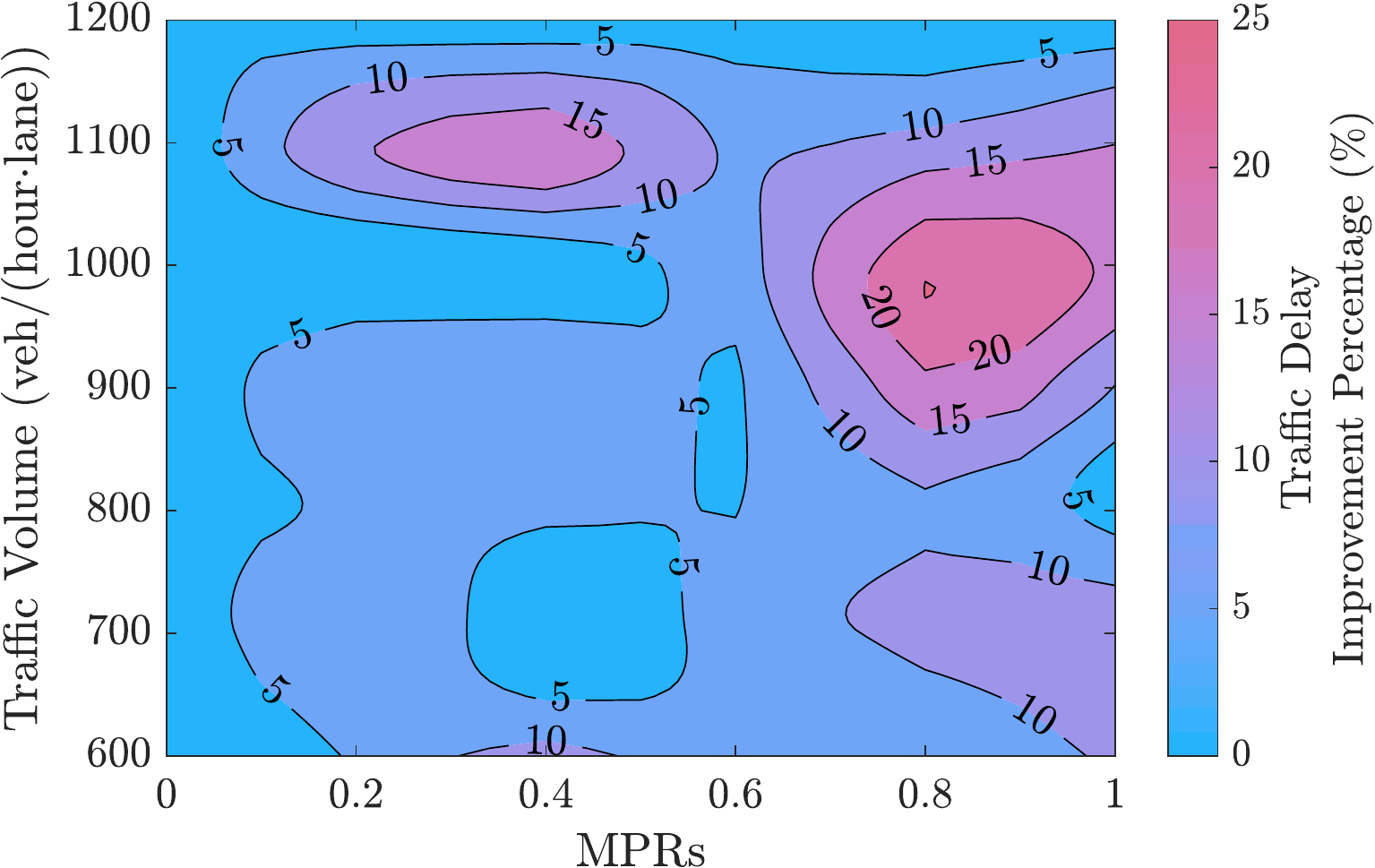}}
	\hfil
	\subfigure[Fuel Consumption Improvement in 2D\label{fig:FC:2D}] {\includegraphics[width=0.45\linewidth]{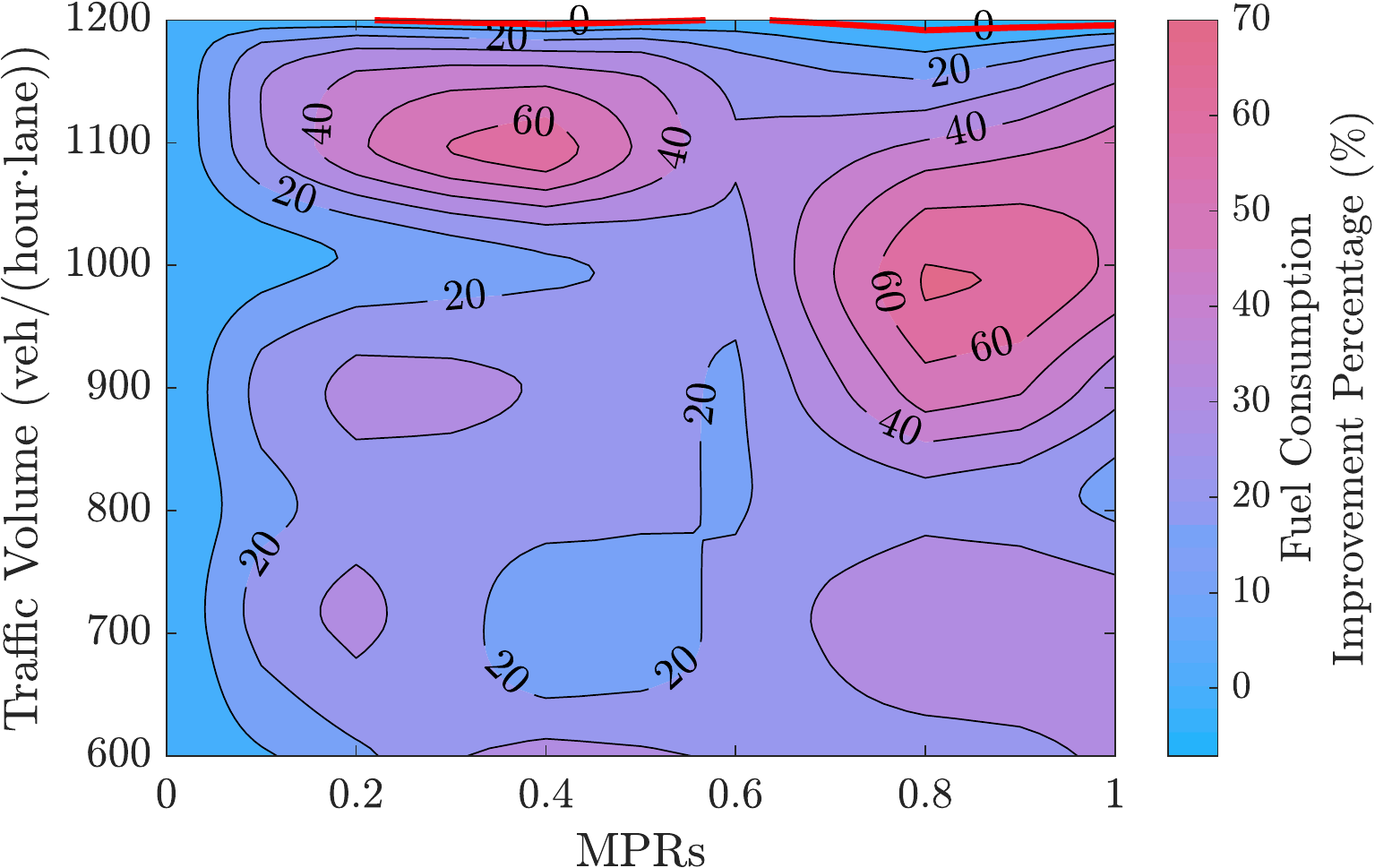}}
	\caption{Algorithm performance comparison in different traffic volumes and different MPRs. $ 0\% $ MPR is chosen as the reference value to evaluate the MP algorithm traffic benefits, which is shown as the red lines in Fig.~\ref{fig:ATTD:3D} $\sim$ Fig.~\ref{fig:FC:2D}.}
\end{figure*}

	%
	
	The specific results of ATTD and fuel consumption under MP are shown in Table~\ref{tab:Result:ATTD} and Table~\ref{tab:Result:FC} in~\ref{sec:Appendix_Simulation}. As can be observed, ATTD and fuel consumption get worse as traffic volume increase while the penetration of CAV helps improve the traffic mobility and fuel consumption. To present the result in a more straightforward way, the $ 0\% $ MPR is chosen as the reference value to evaluate the MP algorithm improvement, \emph{i.e.}, how much improvement MP algorithm brings compared to the case where all the vehicles are HDVs. As shown in Fig.\ref{fig:ATTD:3D} and~\ref{fig:FC:3D}, the red lines in the figure denote the reference plane of $ 0\% $ MPR and the surface means the MP algorithm improvements on the ATTD or fuel consumption. Similarly, the figures are re-plotted as 2D contours as shown in Fig.\ref{fig:ATTD:2D} and~\ref{fig:FC:2D}. It is shown that in general, the MP algorithm can improve the traffic efficiency and fuel consumption at the same time in $600\mathrm{-}1200 \, \mathrm{veh/(hour} \cdot \mathrm{lane})$. The highest improvement occurs at around $ 35{-}40\% $ MPR in $1100\, \mathrm{veh/(hour} \cdot \mathrm{lane})$ and $ 80{-}85\% $ MPR in $1000\, \mathrm{veh/(hour} \cdot \mathrm{lane})$, with a $ 20\% $ improvement in ATTD and a $ 60\% $ improvement in fuel consumption. 
	
	We then make further investigations on the influence of traffic volumes. It is observed that the highest improvement of MP algorithm occurs at $1000\mathrm{-}1100$ $\mathrm{veh/(hour} \cdot \mathrm{lane})$, which meets our expectation on the algorithm. Given the maximum velocity $ v_{\max} $, it takes the first CAV $  L_{\mathrm{obs}}/v_{\max}  $ time to arrive at CZ. At that time, the CAV is firstly optimized and the optimal following vehicle number $ n$ is obtained by~\eqref{equ2:T_Green:N}. Accordingly, in the best case, the number of the following vehicles perfectly matches our design, and the corresponding traffic density is $  L_{\mathrm{obs}}/(v_{\max} \cdot n)  $. According to Table~\ref{tab3:OptimalControlMode} and Fig.~\ref{fig:DistanceNumber}, it can be obtained that the best traffic density is $ 3.3 \, \mathrm{s/veh} = 1080 \, \mathrm{veh/(hour} \cdot \mathrm{lane})$ in our simulation scenario, which is consistent with our simulation results.
	When traffic volume reaches $1200 \, \mathrm{veh/(hour} \cdot \mathrm{lane})$, the MP algorithm can hardly improve the traffic efficiency and fuel consumption, since in this high-density traffic condition, the queuing is inevitable and the optimization space for CAVs is rather small. 
	
	Finally, we analyze the influence of MPR. Note that existing research has pointed out that as MPR increases, the CAV's benefits on intersections may not increase in positive correlation; see, \eg,~\cite{ala2016modeling}. 
	A similar result is also observed in our simulation. The proposed MP algorithm separates the mixed traffic flow into standard ``$ 1+n $" mixed platoons. The most crucial factor that influences the performance is whether all the HDVs are included in the mixed platoons. When MPR is relatively high (\eg, $ 70\% $), almost all the HDVs can be incorporated, which helps to greatly improve the traffic performance. However, when the MPR is extremely high (\eg, $ > 90\% $), the improvement percentage drops a little, which means that the MP algorithm is not the best policy for almost 100\% MPR scenarios. In general, our proposed algorithm shows evident positive improvement in almost all traffic conditions. The extensive simulation results confirm the great benefits of the MP algorithm in the mixed traffic intersection.

	\section{Conclusions}
	\label{Section:Conclusion}
	
	In this paper, the notion of ``$ 1+n $" mixed platoon for CAV control in mixed traffic intersections has been proposed. Assigned as the leading vehicle of $ n $ HDVs, the potential of the CAV in improving the global traffic mobility at the intersection has been revealed. Based on rigorous theoretical analysis, we have pointed out that the proposed mixed platoon formation is open-loop stable and controllable in a very mild condition, which is irrelevant to the mixed platoon size $ n $. 
	An optimal control framework has been established with consideration of velocity deviation and fuel consumption of the whole mixed platoon, where the terminal velocity has also been optimized to improve traffic throughput. Furthermore, a hierarchical event-triggered algorithm has been designed to solve the collision problem between adjacent mixed platoons, which can be applied in any MPR of mixed traffic environments. Traffic simulations have verified the effectiveness of the proposed optimal control method.
	
	
	{One future direction is to address the problem of possible heterogeneous dynamics and model uncertainties in HDVs' behaviors in the optimal control framework for the ``$1+n$'' mixed platoon. 
	Considering that multiple CAVs could exist in a mixed platoon when passing the intersection in the same green phasing time, especially in a high-MPR scenario, another interesting topic is to apply cooperative control algorithms to CAVs to further improve the overall performance for the mixed traffic intersection. Moreover, note that this paper focuses on the longitudinal control of CAVs by forbidding lane changing in CZ, addressing the lane changing behavior formally for both CAVs and HDVs is also a significant future direction.} Finally, field experiments are also needed for further validation of the proposed algorithm.
	
	
	\section*{Acknowledgments}
	This work was supported by the National Key Research and Development Program of China under Grant\\ 2018YFE0204302, the National Natural Science Foundation of China under Grant 52072212, the Key-Area Research and Development Program of Guangdong Province under Grant 2019B090912001, Institute of China Intelligent and Connect Vehicles (CICV) and Tsinghua University-Didi Joint Research Center for Future Mobility. 
	
		\appendix
		\section{Proof of Theorem~\ref{Th:Stability}}
		\label{sec:Appendix}
		We consider the ``$1+n$'' mixed platoon stability consists of $ 1 $ leading CAV and $ n $ following HDVs. Firstly, we consider the initial case of $ n=1 $, \emph{i.e.}, there is only one HDV following the leading CAV. 
		To obtain the eigenvalue $\lambda$ of $A_{n = 1}$, we need to solve $ \left| \lambda I - A_{n = 1} \right| = 0 $, leading to
		$
		{\lambda}^{4} + {\alpha_{2}}{\lambda}^{3} + {\alpha_{1}} {\lambda}^{2} = 0,
		$
		whose solutions are
		\begin{equation}
			\label{equ2:n2:EigenValue}
			\begin{gathered}
				\lambda_{1} =0,\; \lambda_{2} =0,\\
				\lambda_{3} =-\frac{\alpha_{2}}{2} - \frac{\sqrt{\alpha_{2}^{2} - 4 \alpha_{1}}}{2},\;
				\lambda_{4} =-\frac{\alpha_{2}}{2} + \frac{\sqrt{\alpha_{2}^{2} - 4 \alpha_{1}}}{2}.
			\end{gathered}	
		\end{equation}
		The two zero eigenvalues indicate that the system is not asymptotically stable, but can be made critically stable. According to Definition~\ref{def:Stability}, we have the stability criterion as $\alpha_{1}>0, \alpha_{2}>0$.
		
		Then, we assume the system is stable when $ n=k $, \ie, all the eigenvalues of $A_{n=k}$ have negative real parts, and focus on the case where $ n=k+1 $. The system matrix can be expressed as
		\begin{equation}
			A_{n = k + 1}=\left[ 
			\begin{array}{ccc}
				{{A}_{n = k}} & {0} & {0}\\
				{1} & {0} & {-1}\\
				{\alpha_{3}} & {\alpha_{1}} & {-\alpha_{2}}
			\end{array}\right].
		\end{equation}
		Given a block matrix 
		$$ 
		P = \left[ \begin{array}{cc}
			A & B \\
			C & D
		\end{array}\right],
		$$
		if $ D $ is invertible, it holds that\cite{greub2012linear}
		$$
		|P| = \left| \begin{array}{cc}
			A & B \\
			C & D
		\end{array}\right| = \left|A-BD^{-1}C\right|\left|D\right| .
		$$
		According to the initial state assumption, we have that 
		$$ 
		\begin{bmatrix}
			{\lambda} & {1}\\
			{-\alpha_{1}} & {\lambda + \alpha_{2}}
		\end{bmatrix}
		$$ 
		is invertible, and thus it is obtained that
		\begin{equation}
			\left| \lambda I - {A}_{n = k + 1} \right| 
			= \left|\lambda I - {A}_{n = k}\right|\left|\begin{array}{cc}
				{\lambda} & {1}\\
				{-\alpha_{1}} & {\lambda + \alpha_{2}}
			\end{array}\right|
			=\left|\lambda I - {A}_{n = k}\right| (\lambda^2+
			\alpha_2 \lambda+\alpha_1).
		\end{equation}
		According to the assumption, we only need to consider the solutions of $ \lambda^2+
		\alpha_2 \lambda+\alpha_1=0$, which are stable if and only if $\alpha_{1}>0, \alpha_{2}>0$. Hence, we have that ${A}_{n = k + 1} $ is stable if and only if $\alpha_{1}>0, \alpha_{2}>0$, which completes the proof of Theorem~\ref{Th:Stability} according to the method of mathematical induction.
	
	\section{Performance Indexes of Simulations under Different Traffic Volumes and MPRs}
	\label{sec:Appendix_Simulation}
	
	In this section, we present the specific data of performance indexes under MP control at different traffic volumes and MPRs in Section~\ref{sec:FinalSimulation}.
	\begin{table}[h]
	\centering
	\begin{tabular}{cccccccc}
		\toprule
		\diagbox[width=12em]{\textbf{MPR}}{\textbf{Traffic volume} \\ $ \mathbf{(vphpl)} $} & $ \mathbf{600} $ & $ \mathbf{700} $ & $ \mathbf{800} $ & $ \mathbf{900} $ & $ \mathbf{1000} $ & $ \mathbf{1100} $ & $ \mathbf{1200} $ \\
		\midrule
		{$ \mathbf{0\%} $}   & $ 92.54 $ &  $ 114.86 $ & $ 132.10 $ & $ 151.22 $ & $ 163.67 $ & $ 178.98 $ & $ 178.92 $  \\
		\cline{2-8}
		{$ \mathbf{20\%} $} & $ 92.43 $ & $ 98.51 $ & $ 130.26 $ & $ 137.99 $ & $ 153.73 $ & $ 162.17 $ & $ 169.13 $ \\
		{$ \mathbf{40\%} $} & $ 92.31 $ & $ 93.12 $ & $ 128.53 $ & $ 139.12 $ & $ 154.69 $ & $ 175.66 $ & $ 156.51 $ \\
		{$ \mathbf{60\%} $} & $ 92.11 $ & $ 104.69 $ & $ 122.42 $ & $ 144.64 $ & $ 155.61 $ & $ 167.25 $ & $ 168.55 $ \\
		{$ \mathbf{80\%} $} & $ 92.27 $ & $ 101.41 $ & $ 100.20 $ & $ 122.91 $ & $ 148.48 $ & $ 158.63 $ & $ 165.96 $ \\
		{$ \mathbf{100\%} $} & $ 92.27 $ & $ 97.75 $ & $ 107.10 $ & $ 136.39 $ & $ 158.79 $ & $ 156.97 $ & $ 159.95 $ \\
		\bottomrule
	\end{tabular}
	\caption{Average travel time delay (s) simulation results of $ 0\% $ and $ 100\% $ MPR under $ 600 \, \mathrm{to} \, 1200 \, \mathrm{veh/(hour \cdot lane)}$ traffic volume.}
	\label{tab:Result:ATTD}
\end{table}

\begin{table}[h]
	\centering
	\begin{tabular}{cccccccc}
		\toprule
		\diagbox[width=12em]{\textbf{MPR}}{\textbf{Traffic volume} \\ $ \mathbf{(vphpl)} $} & $ \mathbf{600} $ & $ \mathbf{700} $ & $ \mathbf{800} $ & $ \mathbf{900} $ & $ \mathbf{1000} $ & $ \mathbf{1100} $ & $ \mathbf{1200} $ \\
		\midrule
		{$ \mathbf{0\%} $}   & $ 3.81 $ &  $ 12.2 $ & $ 21.66 $ & $ 35.29 $ & $ 45.68 $ & $ 59.64 $ & $ 58.23 $  \\
		\cline{2-8}
		{$ \mathbf{20\%} $} & $ 3.78 $ & $ 6.28 $ & $ 19.17 $ & $ 23.57 $ & $ 34.37 $ & $ 40.31 $ & $ 48.12 $ \\
		{$ \mathbf{40\%} $} & $ 4.06 $ & $ 4.18 $ & $ 17.59 $ & $ 24.94 $ & $ 35.25 $ & $ 51.31 $ & $ 37.42 $ \\
		{$ \mathbf{60\%} $} & $ 3.76 $ & $ 8.45 $ & $ 16.19 $ & $ 28.80 $ & $ 36.78 $ & $ 46.57 $ & $ 42.07 $ \\
		{$ \mathbf{80\%} $} & $ 4.04 $ & $ 7.20 $ & $ 6.47 $ & $ 15.61 $ & $ 33.20 $ & $ 36.47 $ & $ 44.87 $ \\
		{$ \mathbf{100\%} $} & $ 4.08 $ & $ 5.79 $ & $ 9.07 $ & $ 23.32 $ & $ 37.32 $ & $ 38.60 $ & $ 39.12 $ \\
		\bottomrule
	\end{tabular}
	\caption{Fuel consumption ($ \mathrm{L}/100\mathrm{km} $) simulation results of $ 0\% $ and $ 100\% $ MPR under $ 600 \, \mathrm{to} \, 1200 \, \mathrm{veh/(hour \cdot lane)}$ traffic volume.}
	\label{tab:Result:FC}
\end{table}
	
	
	\bibliography{mybibfile}

\end{document}